\documentclass[12pt,a4paper]{article}
\usepackage{amsfonts,amsthm,amsmath,amssymb,amscd}
\usepackage{fancybox}
\usepackage[latin1]{inputenc}
\usepackage[usenames]{color}
\usepackage{graphicx}
\usepackage{enumerate}
\usepackage{float}
\usepackage{fullpage}
\usepackage{epsfig,amssymb} 
\usepackage{subcaption}

\newtheorem{theorem}{Theorem}[section]
\newtheorem{corollary}[theorem]{Corollary}
\newtheorem{lemma}[theorem]{Lemma}
\newtheorem{proposition}[theorem]{Proposition}
\newtheorem{definition}[theorem]{Definition}

\newtheorem{rem}[theorem]{Remark}

\begin{document}
	
\title{\textbf{Renormalized Area of Hypersurfaces in Hyperbolic Spaces}}

\author{A. P\'ampano}
\date{\today}

\maketitle

\begin{abstract}
We employ Chen's conformal invariant quantity \cite[Theorem 1]{C} in combination with the Chern-Gauss-Bonnet formulas to obtain expressions for the renormalized area of asymptotically minimal hypersurfaces in the $(2n+1)$-dimensional hyperbolic space $\mathbb{H}^{2n+1}$, $n=1,2$. Our results extend Alexakis and Mazzeo's formula for the renormalized area for surfaces in $\mathbb{H}^3$ \cite[Proposition 3.1]{AM} as well as their relation between the renormalized area of minimal surfaces of $\mathbb{H}^3$ and the Willmore energy of their doubles in $\mathbb{R}^3$ \cite[Proposition 8.1]{AM} to the non-minimal case and to the higher dimensional case $n=2$.

Moreover, we also generalize our results by considering hypersurfaces in $(2n+1)$-dimensional Poincar\'e-Einstein spaces and even-dimensional submanifolds of arbitrary codimension.\\

\end{abstract}

\section{Introduction}

Originated one century ago with the works of H. Weyl and E. Cartan, among others, the theory of conformal differential geometry is the branch of differential geometry which studies conformal invariant quantities, that is, geometric quantities that remain invariant under conformal transformations. Two such quantities that have received a lot of attention in the last decades are the (conformal) bending energy and the renormalized area of minimal hypersurfaces.

The classical bending energy of surfaces arises from the elastic theory pioneered by S. Germain and S. Poisson. Indeed, it is a particular case of the curvature-dependent integrals suggested by S. Germain at the beginning of the 19th century to study the physical system associated with Chladni plates. For a surface $M$ immersed in the Euclidean space $\mathbb{R}^3$, the (conformal) bending energy is presented as
$$\mathcal{W}(M)=\int_M\left(\overline{H}^2-\overline{K}\right)d\overline{A}\,,$$ 
where $\overline{H}$ and $\overline{K}$ represent the mean and Gaussian curvatures of the surface $M$, respectively. Later reintroduced by T. Willmore \cite{W}, the (conformal) bending energy $\mathcal{W}$ is usually referred to as the Willmore functional. The conformal invariance of the Willmore functional is an essential property shown by W. Blaschke and G. Thomsen \cite{BT} in the early 20th century and which has proven itself very useful to understand the shape of equilibria and minimizers. In 1974, B.-Y. Chen \cite{C} extended this conformal invariant quantity to submanifolds $M$ of arbitrary Riemannian manifolds $N$. More precisely, in Theorem 1 of \cite{C}, he showed that the quantity
$$\left(H^2-R\right)g\,,$$
is invariant under any conformal change of metric. In above expression, $H$ and $R$ denote, respectively, the mean and extrinsic scalar curvatures of $M$ and $g$ is the Riemannian metric of the ambient space $N$.

On the other hand, motivated by the physical theory of the AdS/CFT correspondence which had been recently proposed by J. Maldacena \cite{M} (see also \cite{Wi}) and, in particular, by the paper \cite{M1}, C. R. Graham and E. Witten \cite{GW} introduced in 1999 the renormalized area functional acting on a certain class of minimal submanifolds of Poincar\'e-Einstein spaces (these spaces were presented in \cite{FG}). Among other things, they showed that the renormalized area is well defined when the submanifold is even-dimensional. The specific case of minimal surfaces was further analyzed by S. Alexakis and R. Mazzeo \cite{AM}, who succeeded to obtain a explicit Gauss-Bonnet type formula for their renormalized area, namely,
$$\mathcal{A}_R(M)=-2\pi\chi(M)-\frac{1}{2}\int_M \lvert\mathring{B}\rvert^2\, dA\,.$$
Here, $\mathcal{A}_R(M)$ stands for the renormalized area of a minimal surface $M$ properly embedded in a hyperbolic $3$-space, $\mathring{B}$ is the trace-free second fundamental form of $M$, and $\chi(M)$ represents the Euler characteristic of $M$. In the case of a $4$-dimensional minimal hypersurface of a Poincar\'e-Einstein space, the analogue of the above formula has been recently presented in Theorem 1.1 of \cite{T}.

In the paper \cite{AM}, S. Alexakis and R. Mazzeo employed their Gauss-Bonnet type formula to deduce that the renormalized area of a \emph{minimal} surface in the hyperbolic space $\mathbb{H}^3$ coincides, up to a coefficient $-2$, with the Willmore functional of its double regarded as a surface in $\mathbb{R}^3$. This was the first time a connection was drawn between these two quantities. However, even though in the same paper \cite{AM} it was shown that the renormalized area may be considered for a larger family of surfaces (not necessarily minimal), it was noted that it was not clear if such a simple connection between the renormalized area of (admissible) non-minimal surfaces and the Willmore functional existed. In a recent paper \cite{PP}, B. Palmer and the author positively answered this open question, obtaining the relation (see also Theorem \ref{t2} below)
$$\mathcal{A}_R(M)=\int_M H^2\,dA-\int_M \overline{H}^2\,d\overline{A}\,,$$
where $H$ is the mean curvature of the surface $M$ in the hyperbolic space $\mathbb{H}^3$ and $dA$ is the area element on $M$ induced from the hyperbolic metric ($\overline{H}$ and $d\overline{A}$ denote the analogues, when $M$ is regarded after a conformal change of metric as a surface in $\mathbb{R}^3$). The above general relation has arisen in \cite{PP} while the authors analyzed a second order reduction of the Helfrich equation, which represents capillary surfaces in $\mathbb{H}^3$ under the action of constant gravity. The proof is based on the divergence theorem and despite the fact that one can mimic it for higher dimensional cases, it is not obvious at all how to infer information regarding the renormalized area.

Even though both quantities discussed above have been widely studied independently, explicit treatments directly relating them, or justifying their relationship, appear to be scarce. In this paper, we show that both are closely related through the classical Gauss-Bonnet theorem and its higher dimensional analogues, the Chern-Gauss-Bonnet formulas. More precisely, beginning with Chen's conformal invariant quantity \cite[Theorem 1]{C}, we obtain a formula for the area of a hypersurface $M$ of the $(2n+1)$-dimensional hyperbolic space $\mathbb{H}^{2n+1}$ in terms of the Euclidean (after conformally changing the metric) and hyperbolic bendings (Proposition \ref{t1}). When the hypersurface $M$ is asymptotically minimal (i.e., $H=\mathcal{O}(z^n)$), which in particular means that $M$ approaches the ideal boundary $\partial_\infty \mathbb{H}^{2n+1}$ at a right angle, this formula for the hyperbolic area can be used to understand the renormalized area of $M$. To do so, we employ the Chern-Gauss-Bonnet formulas in order to simplify some of the integral terms arising in the above-mentioned expression. We explicitly carry out all the computations for the lower dimensional cases of surfaces in $\mathbb{H}^3$ and hypersurfaces in $\mathbb{H}^5$, obtaining explicit expressions for their renormalized areas, regardless if they are minimal or not (see Theorems \ref{t2} and \ref{t3}, respectively). Of course, for surfaces in $\mathbb{H}^3$ we recover the relation of \cite{PP}, although our proof is essentially different. For the case of hypersurfaces in $\mathbb{H}^5$ our result shows that the renormalized area can be defined for asymptotically minimal ($H=\mathcal{O}(z^2)$) hypersurfaces, a strictly larger class than that of minimal ones.

Applying once again Theorem 1 of \cite{C} and the Chern-Gauss-Bonnet theorem to our general formulas for the renormalized area of hypersurfaces in $\mathbb{H}^{2n+1}$, $n=1,2$, we obtain Gauss-Bonnet type expressions for $\mathcal{A}_R$ which extend those for the minimal case presented in \cite{AM} and \cite{T}, respectively (see Corollaries \ref{c1} and \ref{c2} below). In particular, the Gauss-Bonnet type formula for the renormalized area of surfaces in $\mathbb{H}^3$ reads (c.f., Proposition 3.1 of \cite{AM})
$$\mathcal{A}_R(M)=-2\pi\chi(M)-\frac{1}{2}\int_M \lvert\mathring{B}\rvert^2\,dA+\int_M H^2\,dA\,.$$
The renormalized area in this expression is decomposed in three terms: a topological term which is a multiple of the Euler characteristic, a conformal invariant term consisting of the integral of Chen's quantity (c.f., Remark \ref{rem} below), and an extrinsic non-conformally invariant term given by the (hyperbolic) bending. Such a decomposition also holds for hypersurfaces in $\mathbb{H}^5$ (see Corollary \ref{c2}).

In the last two sections, we extend our Gauss-Bonnet type formulas for the renormalized area of hypersurfaces to the more general ambient spaces given by Poincar\'e-Einstein manifolds (Section 5) and to arbitrary codimension (Section 6). The approach we present is simpler than the one of Corollaries \ref{c1} and \ref{c2}, in the sense that it does not involve any formulas of $\mathcal{A}_R$ in terms of the difference of the bendings. Nonetheless, as in the hyperbolic case, both the Chern-Gauss-Bonnet formulas and Chen's conformal invariant quantity play a central role. Our computations can be extended to higher dimensional cases in a fairly straightforward manner.

\section{Basic Definitions and Notation}

Let $n$ be a fixed positive integer and denote by $N^{2n+1}(\rho)$ the $(2n+1)$-dimensional space form of constant sectional curvature $\rho\in\mathbb{R}$. If $\rho=0$, the space form $N^{2n+1}(\rho)$ is the Euclidean space $\mathbb{R}^{2n+1}$. In terms of the standard coordinates $(x_1,...,x_{2n},z)$, the metric of $\mathbb{R}^{2n+1}$ is given by
$$\overline{g}=\sum_{i=1}^{2n} dx_i^2+dz^2\,.$$
The $(2n+1)$-dimensional hyperbolic space $\mathbb{H}^{2n+1}$ is the space form $N^{2n+1}(\rho)$ of sectional curvature $\rho=-1$. Throughout this paper, we will consider the upper half-space model for $\mathbb{H}^{2n+1}$. That is, 
$$\mathbb{H}^{2n+1}=\{(x_1,...,x_{2n},z)\in\mathbb{R}^{2n+1}\,\lvert \, z>0\}\,,$$
endowed with the (hyperbolic) metric
$$g=\frac{1}{z^2}\left(\sum_{i=1}^{2n}dx_i^2+dz^2\right)=\frac{1}{z^2}\,\overline{g}.$$
The ideal boundary of the hyperbolic space $\mathbb{H}^{2n+1}$ is the set of points in $\mathbb{R}^{2n+1}$ with $z=0$, i.e.,
$$\partial_\infty\mathbb{H}^{2n+1}=\{(x_1,...,x_{2n},0)\in\mathbb{R}^{2n+1}\}\,.$$

Let $M$ be a smooth and compact (with or without boundary) hypersurface embedded in $N^{2n+1}(\rho)$ (hence, $M$ is a $2n$-dimensional manifold itself). Pick up a choice of unit normal vector field $\xi$ to $M$ and denote by $B$ the second fundamental form of $M$ in $N^{2n+1}(\rho)$. The principal directions of $M$ are denoted by $v_1,...,v_{2n}$, while $\kappa_1,...,\kappa_{2n}$ represent the associated principal curvatures. The mean curvature $H$ and the extrinsic scalar curvature $R$ of the hypersurface $M$ are defined in terms of the principal curvatures as
\begin{equation}\label{HandR}
	H=\frac{1}{2n}\sum_{i=1}^{2n}\kappa_i\,,\quad\quad\quad R=\frac{1}{2n(2n-1)}\sum_{i\neq j}\kappa_i\kappa_j\,,
\end{equation}
respectively. For future use, it is convenient to express $H$ and $R$ in terms of the second fundamental form $B$. Clearly, from \eqref{HandR}, $H={\rm Trace}B/(2n)$, while expanding $(\kappa_1+\cdots+\kappa_{2n})^2$ we can check that $R$ satisfies
\begin{equation}\label{RinB}
	R=\frac{1}{2n(2n-1)}\left(4n^2 H^2-\lvert B\rvert^2\right),
\end{equation}
where $\lvert B\rvert^2=\kappa_1^2+\cdots+\kappa_{2n}^2$ is the square of the second fundamental form $B$. Observe also that $\lvert B\rvert^2=2nC$, where $C$ is the Casorati curvature \cite{BY} and, hence, $R=(2nH^2-C)/(2n-1)$.

Employing the Gauss' equation, the extrinsic scalar curvature $R$ of $M$ can be related to the sectional curvature $\rho$ of the ambient space $N^{2n+1}(\rho)$ by
\begin{equation}\label{extrinsic-intrinsic}
	R=\lambda-\rho\,,
\end{equation}
where $\lambda$ is the (intrinsic) scalar curvature of $M$ defined as
\begin{equation}\label{lambda}
	\lambda=\frac{1}{2n(2n-1)}\sum_{i\neq j} K(e_i,e_j)=\frac{1}{2n(2n-1)}\sum_{i\neq j} {\rm Rm}(e_i,e_j,e_i,e_j)\,.
\end{equation}
Here, $K(e_i,e_j)$ denotes the sectional curvature of the planar section spanned by $e_i$ and $e_j$, where $\{e_1,...,e_{2n}\}$ is any orthonormal frame tangent to $M$, and ${\rm Rm}$ stands for the Riemann curvature of $M$. Observe that, in order to simplify upcoming expressions, we have chosen to include a coefficient depending on the dimension in our definition of scalar curvature, contrary to the standard case. 

The area and bending of the hypersurface $M$ in $N^{2n+1}(\rho)$ are, respectively,
\begin{equation}\label{area-bending}
	\mathcal{A}(M)=\int_M 1\,dA\,,\quad\quad\quad \mathcal{B}(M)=\int_M H^{2n}\,dA\,,
\end{equation}
where $dA$ denotes the volume element on $M$ induced from the metric of the ambient space form $N^{2n+1}(\rho)$. If $M$ is a surface in $\mathbb{R}^3$, the bending of $M$ defined above differs from the Willmore functional $\mathcal{W}$ by the total Gaussian curvature term.

We also recall here the Chern-Gauss-Bonnet formula for $2n$-dimensional manifolds $M$ with boundary, in the lower dimensional cases $n=1,2$. When $M$ is a surface (that is, $n=1$), the Chern-Gauss-Bonnet formula reduces to the classical Gauss-Bonnet theorem
\begin{equation}\label{GB}
	\int_MK\,dA=2\pi\chi(M)-\oint_{\partial M}\kappa_g\,ds\,,
\end{equation}
where $K$ is the Gaussian curvature of $M$, $\kappa_g$ is the geodesic curvature of $\partial M$, and $\chi(M)$ denotes the Euler characteristic of the surface (with boundary) $M$. If $n=2$, $M$ is a $4$-dimensional manifold with boundary, in which case the Chern-Gauss-Bonnet formula can be written as (see, for instance, \cite{CC,CS} and references therein)
\begin{equation}\label{CGB0}
	\int_M 2\sigma_2(P)\,dA=4\pi^2\chi(M)-\frac{1}{8}\int_M\lvert W\rvert^2\,dA-\oint_{\partial M}S\,ds\,,
\end{equation}
where $\sigma_2(P)$ is the second elementary symmetric polynomial of the eigenvalues of the Schouten tensor $P$, $W$ is the Weyl tensor of $M$, and the $S$-curvature is given by
\begin{equation}\label{Scurvature}
	S=6\lambda h-{\rm Ric}(\nu,\nu)h-\sum_{j,k=1}^{4}{\rm Rm}(e_i,e_j,e_i,e_k)L^{jk}+\frac{1}{3}h^3-h\lvert L\rvert^2+\frac{2}{3}{\rm Trace}\left(L^3\right).
\end{equation}
Here, $L$ represents the second fundamental form of $\partial M$ in $M$, $h$ is the trace of $L$, ${\rm Ric}$ is the Ricci tensor, and $\nu$ is the inward-pointing normal to $\partial M$ in $M$.

From the definition of the Schouten tensor $P$, the polynomial $\sigma_2$ (see the Introduction of \cite{CC} for details), and the relation between the square of the Ricci tensor and its trace-free tensor, we obtain that the integrand $2\sigma_2(P)$ on the left-hand side of \eqref{CGB0} can be expressed in terms of the scalar curvature $\lambda$ and the square of the trace-free Ricci tensor $\lvert E\rvert^2$ as
\begin{equation}\label{sigma2}
	2\sigma_2(P)=3\lambda^2-\frac{1}{4}\lvert E\rvert^2\,.
\end{equation}
Thus, the Chern-Gauss-Bonnet formula \eqref{CGB0} for $4$-dimensional manifolds $M$ with boundary can be described in the equivalent form
\begin{equation}\label{CGB}
	\int_M \lambda^2\,dA=\frac{4}{3}\pi^2\chi(M)-\frac{1}{24}\int_M\lvert W\rvert^2\,dA+\frac{1}{12}\int_M\lvert E\rvert^2\,dA-\frac{1}{3}\oint_{\partial M}S\,ds\,.
\end{equation}

From now on, we will consider that the hypersurface $M$ is embedded in $\mathbb{H}^{2n+1}$. After a conformal change of metric given by $\overline{g}=z^2 g$, $M$ can be regarded as a hypersurface in $\mathbb{R}^{2n+1}$. In this setting, we will denote with an upper bar the Euclidean objects associated to $M$, to distinguish them from the hyperbolic ones (for instance, $H$ will denote the mean curvature of $M$ as a hypersurface in $\mathbb{H}^{2n+1}$, while $\overline{H}$ will represent the Euclidean counterpart). There are well known relations between the hyperbolic and Euclidean objects. As an example, for every $i=1,...,2n$,
$$\kappa_i=z\,\overline{\kappa_i}+\overline{\xi}_z\,,$$
where $\kappa_i$ and $\overline{\kappa}_i$ are the principal curvatures of $M$ regarded as a hypersurface in $\mathbb{H}^{2n+1}$ and $\mathbb{R}^{2n+1}$, respectively; while $\overline{\xi}_z$ is the last component of the Euclidean unit normal vector field to $M$. It then follows from \eqref{HandR} that
\begin{equation}\label{HRhyp-euc}
	H=z\,\overline{H}+\overline{\xi}_z\,,\quad\quad\quad R=z^2\overline{R}+2z\,\overline{\xi}_z\overline{H}+\overline{\xi}_z^2\,.
\end{equation} 
We refer the reader to \cite{C} for more relations and details.

\section{Formula for the Area}

Let $M$ be a smooth and compact hypersurface embedded in $\mathbb{H}^{2n+1}$. Specializing Theorem 1 of \cite{C} to our case, we deduce that the quantity
$$\left(H^2-R\right)g\,,$$
is conformally invariant. Hence, conformally changing the metric according to $\overline{g}=z^2g$ and understanding $M$ as a hypersurface in $\mathbb{R}^{2n+1}$, we obtain the equality
\begin{equation}\label{equality}
	\left(H^2-\lambda-1\right)g=\left(H^2-R\right)g=\left(\overline{H}^2-\overline{R}\right)\overline{g}=\left(\overline{H}^2-\overline{\lambda}\right)\overline{g}\,,
\end{equation}
where we have used the relation \eqref{extrinsic-intrinsic} between the intrinsic and extrinsic scalar curvatures of $M$ (both as a hypersurface in $\mathbb{R}^{2n+1}$ and $\mathbb{H}^{2n+1}$). This equality directly follows from \eqref{HRhyp-euc} as well.

\begin{rem}\label{rem} Observe that Chen's conformal invariant quantity $(H^2-R)\,g$ can be expressed as a multiple of the square of the trace-free second fundamental form $\lvert \mathring{B}\rvert^2$ times the metric $g$. Indeed, from \eqref{RinB} and the standard relation $\lvert B\rvert^2=\lvert\mathring{B}\rvert^2+2n H^2$, we get
	\begin{equation}\label{ChenB}
		H^2-R=\frac{1}{2n(2n-1)}\,\lvert\mathring{B}\rvert^2\,.
	\end{equation}
\end{rem}

As a straightforward consequence of the equality \eqref{equality}, we will next deduce a formula for the area of $M$ as a hypersurface in $\mathbb{H}^{2n+1}$.

\begin{proposition}\label{t1}
	Let $M$ be a compact hypersurface in $\mathbb{H}^{2n+1}$ with mean curvature $H$, scalar curvature $\lambda$, and extrinsic scalar curvature $R$. The area of $M$ satisfies
	\begin{eqnarray}\label{formulaarea}
		\mathcal{A}(M)&=&\sum_{i=0}^{n-1}(-1)^{n+i} \begin{pmatrix} n \\ i \end{pmatrix} \left( \int_M\overline{H}^{2(n-i)}\overline{R}^i\,d\overline{A}-\int_M H^{2(n-i)}R^i \,dA\right)\nonumber\\
		&&+\int_M \overline{\lambda}^n\, d\overline{A}-\sum_{i=1}^n\begin{pmatrix} n \\ i \end{pmatrix} \int_M\lambda^i\, dA\,,
	\end{eqnarray}
	where $\overline{H}$, $\overline{\lambda}$, and $\overline{R}$ are the mean, the scalar, and the extrinsic scalar curvatures of $M$ regarded as a hypersurface in $\mathbb{R}^{2n+1}$ and $d\overline{A}$ is the volume element on $M$ induced from the Euclidean metric.
\end{proposition}
\textit{Proof.} Let $M$ be a compact hypersurface in $\mathbb{H}^{2n+1}$ and denote by $H$ and $R$ the mean and the extrinsic scalar curvatures, respectively. Consider the conformal change of metric $\overline{g}=z^2 g$ and regard $M$ as a hypersurface in $\mathbb{R}^{2n+1}$ with Euclidean mean curvature $\overline{H}$ and extrinsic scalar curvature $\overline{R}$ (observe that from \eqref{extrinsic-intrinsic} it follows that $\overline{R}=\overline{\lambda}$ where $\overline{\lambda}$ represents the intrinsic scalar curvature).

From the equality \eqref{equality} we deduce that
\begin{equation}\label{square}
	\int_M \left(H^2-R\right)^n\, dA=\int_M\left(\overline{H}^2-\overline{R}\right)^n\, d\overline{A}\,,
\end{equation}
where $dA$ and $d\overline{A}$ are the volume elements on $M$ induced from the hyperbolic and Euclidean metrics, respectively. The integrands on above expression are polynomials. Expanding them, we have
\begin{equation*}
	\left(H^2-R\right)^n= \sum_{i=0}^n (-1)^i \begin{pmatrix} n \\ i \end{pmatrix} H^{2(n-i)} R^i= (-1)^n R^n +\sum_{i=0}^{n-1} (-1)^i \begin{pmatrix} n \\ i \end{pmatrix} H^{2(n-i)} R^i\,.
\end{equation*}
This expansion is exactly the same in the Euclidean setting (with the upper bars), hence we avoid writing it here. In the Euclidean case, we use that $\overline{R}=\overline{\lambda}$ in the first term of the right-hand side. On the other hand, in the hyperbolic setting we employ \eqref{extrinsic-intrinsic} and expand
$$R^n=(\lambda+1)^n=\sum_{i=0}^n \begin{pmatrix} n \\ i \end{pmatrix} \lambda^i=1+\sum_{i=1}^n\begin{pmatrix} n\\i\end{pmatrix}\lambda^i\,.$$
Combining this with the above expansion (for both the Euclidean and hyperbolic settings) and substituting everything in \eqref{square}, we obtain after rearranging
\begin{eqnarray*}
	(-1)^n \mathcal{A}(M)&=&-(-1)^n\sum_{i=1}^n \begin{pmatrix} n \\ i \end{pmatrix} \int_M \lambda^i \,dA- \sum_{i=0}^{n-1}(-1)^i \begin{pmatrix} n \\ i \end{pmatrix}\int_MH^{2(n-i)}R^i \,dA\\
	&&+(-1)^n\int_M\overline{\lambda}^n \,d\overline{A}+\sum_{i=0}^{n-1}(-1)^i \begin{pmatrix} n \\ i \end{pmatrix} \int_M \overline{H}^{2(n-i)}\overline{R}^i \,d\overline{A}\,.
\end{eqnarray*}
Finally, multiplying everything by $(-1)^n$, the relation \eqref{formulaarea} follows. \hfill$\square$

\section{Renormalized Area}

Let $M$ be a compact hypersurface with boundary $\partial M$ and assume that its interior is properly embedded in the upper half-space $\{z>0\}$ of $\mathbb{R}^{2n+1}$. Assume also that the boundary $\partial M$ lies in the hypersurface $\{z=0\}$ meeting this hypersurface at a right angle. Such a hypersurface $M$ can be regarded as a hypersurface in the hyperbolic space $\mathbb{H}^{2n+1}$ which approaches the ideal boundary $\partial_\infty \mathbb{H}^{2n+1}$ at a right angle. Clearly, in such a setting the area of the hypersurface $M$ becomes infinite. However, since $M$ is even-dimensional its area possesses a well defined Hadamard regularization, known as the renormalized area \cite{GW}.

The renormalized area $\mathcal{A}_R$ of $M$ is the constant term in the expansion of the area of the truncated hypersurfaces $M_\epsilon=M\cap\{z\geq \epsilon\}$, for $\epsilon>0$ sufficiently small. For each $\epsilon>0$ fixed, $M_\epsilon$ is a compact hypersurface fully embedded in $\mathbb{H}^{2n+1}$ and, hence, we can employ \eqref{formulaarea} to understand the renormalized area $\mathcal{A}_R$ of $M$. For this purpose, we need to understand the constant terms of the expansions in terms of $\epsilon>0$ of the integrals arising on the right-hand side of \eqref{formulaarea}. More precisely, we have
\begin{eqnarray*}
	\mathcal{A}_R(M)&=&\sum_{i=0}^{n-1}(-1)^{n+i}\begin{pmatrix} n \\ i \end{pmatrix}\left(\int_M \overline{H}^{2(n-i)}\overline{R}^i\,d\overline{A}-\int_M H^{2(n-i)}R^i\,dA\right)+\int_M \overline{\lambda}^n\,d\overline{A}\\
	&&-\sum_{i=1}^{n}\begin{pmatrix} n \\ i \end{pmatrix} {\rm f.p.}\int_M \lambda^i\,dA\,,
\end{eqnarray*}
where ${\rm f.p.}$ refers to the finite part (in other words, the constant term in the expansion in terms of $\epsilon>0$). All the integrals in the first line of the above expression are convergent, since they are Euclidean integrals or because \eqref{HRhyp-euc} in combination with $dA=z^{-2n}d\overline{A}$ and the fact that $M$ approaches the ideal boundary at a right angle. However, those integrals with ${\rm f.p.}$ in front may not be convergent and, hence, it would be desirable to rewrite them in a more suitable way. To this end, we will further restrict the asymptotic behavior of our hypersurfaces.

\begin{definition} We say that the $2n$-dimensional hypersurface $M$ is asymptotically minimal (of order $n$) if $H=\mathcal{O}(z^{n})$.
\end{definition} 

Observe that if $n=1$, from \eqref{HRhyp-euc} we deduce that asymptotically minimal is equivalent to $\overline{\xi}_z=\mathcal{O}(z)$ and so to $M$ meeting the ideal boundary $\partial_\infty\mathbb{H}^{3}$ at a right angle. For $n>1$, the asymptotically minimal hypothesis implies that $M$ approaches $\partial_\infty\mathbb{H}^{2n+1}$ at a right angle, but it is not equivalent to it. Throughout this paper we will assume that $M$ satisfies the above condition.

In the following sections we analyze in detail the cases of surfaces in $\mathbb{H}^3$ and hypersurfaces in $\mathbb{H}^5$ (that is, the cases $n=1,2$).

\subsection{Surfaces in $\mathbb{H}^3$}

We first consider the case in which $M$ is a surface in $\mathbb{H}^3$ (that is, $n=1$ holds). In this specific case, there is a well known formula for the renormalized area \cite[Proposition 3.1]{AM}. In the same paper, this formula was used to relate the renormalized area of a minimal surface $M$ to the Willmore energy of its double (see Proposition 8.1 of \cite{AM}). However, as pointed out just after that result, it was not known if a similar relation holds for non-minimal surfaces $M$.

Here, we employ the area formula \eqref{formulaarea} to extend Alexakis-Mazzeo's relation to the non-minimal case.

\begin{theorem}\label{t2}
	Let $M$ be a surface in $\mathbb{H}^3$ approaching the ideal boundary $\partial_\infty\mathbb{H}^3$ at a right angle. Then, the renormalized area of $M$ is given by the difference between the hyperbolic and Euclidean bendings of $M$. That is,
	\begin{equation}\label{AR3}
		\mathcal{A}_R(M)=\mathcal{B}(M)-\overline{\mathcal{B}}(M)=\int_M H^2 \,dA-\int_M \overline{H}^2\, d\overline{A}\,,
	\end{equation}
	where $H$ is the mean curvature of $M$ as a surface in $\mathbb{H}^3$, $dA$ is the area element on $M$ induced from the hyperbolic metric, while $\overline{H}$ and $d\overline{A}$ are the Euclidean analogues. Moreover, both integrals of \eqref{AR3} are convergent.
\end{theorem}
\textit{Proof.} Let $M$ be a surface in $\mathbb{H}^3$ satisfying the conditions of the statement and $\epsilon>0$ a sufficiently small fixed value. Consider the part of $M$ `far' from the ideal boundary $\partial_\infty\mathbb{H}^3$, i.e., the part $M_\epsilon=M\cap\{z\geq \epsilon\}$. Since $M_\epsilon$ does not approach $\partial_\infty\mathbb{H}^3$, its area is finite and we can use formula \eqref{formulaarea} to obtain
$$\mathcal{A}(M_\epsilon)=\int_{M_\epsilon}H^2\,dA-\int_{M_\epsilon}\overline{H}^2\,d\overline{A}+\int_{M_\epsilon}\overline{\lambda}\,d\overline{A}-\int_{M_\epsilon}\lambda\,dA\,,$$
since $n=1$ holds.

Observe that from our definition \eqref{lambda} of the intrinsic scalar curvature $\lambda$, this curvature for surfaces $M$ is just the sectional curvature $K\equiv K(e_1,e_2)$, in other words, the Gaussian curvature of $M$. We can then apply the classical Gauss-Bonnet theorem \eqref{GB} to deduce that
$$\int_{M_\epsilon}\overline{\lambda}\,d\overline{A}-\int_{M_\epsilon}\lambda\,dA=\left(2\pi\chi(M_\epsilon)-\oint_{\partial M_\epsilon} \overline{\kappa}_g\,d\overline{s}\right)-\left(2\pi\chi(M_\epsilon)-\oint_{\partial M_\epsilon}\kappa_g\,ds\right),$$
where $\chi(M_\epsilon)$ is the Euler characteristic of $M_\epsilon$ and $\kappa_g$ is the geodesic curvature of $\partial M_\epsilon$ ($\overline{\kappa}_g$ is the geodesic curvature with respect to the metric induced from the Euclidean one). Since the topology of $M_\epsilon$ does not change after modifying the metric, both terms involving the Euler characteristic cancel out. Hence, we end up with (after rearranging)
$$\mathcal{A}(M_\epsilon)-\oint_{\partial M_\epsilon}\kappa_g\,ds=\int_{M_\epsilon} H^2 \,dA-\int_{M_\epsilon} \overline{H}^2\,d\overline{A}-\oint_{\partial M_\epsilon}\overline{\kappa}_g\,d\overline{s}\,.$$
From (3.11) of \cite{AM}, it follows that the total geodesic curvature arising on the left-hand side of the above equality has no constant term when considering its expansion in terms of $\epsilon>0$. Hence, it does not contribute to the constant term of the expansion of $\mathcal{A}(M_\epsilon)$, i.e., to the renormalized area $\mathcal{A}_R$ of $M$. Moreover, as $\epsilon\to 0^+$, the (Euclidean) geodesic curvature $\overline{\kappa}_g$ vanishes due to the surface $M$ approaching $\partial_\infty\mathbb{H}^3$ at a right angle. Hence,
$$\oint_{\partial M_\epsilon}\overline{\kappa}_g\,d\overline{s}=\mathcal{O}(\epsilon)\,,$$
and so, it does not contribute to $\mathcal{A}_R$ either. In conclusion, \eqref{AR3} holds.

For the second part of the statement, we note that $\overline{\mathcal{B}}(M)$ is clearly convergent since $d\overline{A}$ is the area element induced from the Euclidean metric. The convergence of the hyperbolic bending $\mathcal{B}(M)$ is a consequence of \eqref{HRhyp-euc}. Indeed, since $M$ meets $\partial_\infty\mathbb{H}^3$ at a right angle, the last component of the Euclidean unit normal $\overline{\xi}_z$ is zero at $\partial M\subset\{z=0\}$. This implies that $\overline{\xi}_z=\mathcal{O}(z)$ and so $H^2=\mathcal{O}(z^2)$ as well. Combining this with $dA=z^{-2}d\overline{A}$ the finiteness of $\mathcal{B}(M)$ follows. \hfill$\square$

\begin{rem} Equation \eqref{AR3} has recently appeared as a particular case in Theorem 3.1 of \cite{PP}. In that work, the result is shown by applying the divergence theorem to the (surface) divergence of $\nabla\log z$. Our approach here is essentially different.
\end{rem}

The general relation \eqref{AR3} can be used, in combination with the Gauss-Bonnet formula \eqref{GB} and Chen's result \cite[Theorem 1]{C}, to obtain the general formula for the renormalized area of surfaces in $\mathbb{H}^3$ first given in Proposition 3.1 of \cite{AM}.

\begin{corollary}\label{c1} Let $M$ be a surface in $\mathbb{H}^3$ approaching the ideal boundary $\partial_\infty \mathbb{H}^3$ at a right angle. Then, the renormalized area of $M$ satisfies
	\begin{equation}\label{GBT3}
		\mathcal{A}_R(M)=-2\pi\chi(M)-\frac{1}{2}\int_M \lvert \mathring{B} \rvert^2\,dA+\mathcal{B}(M)\,,
	\end{equation}
where $\chi(M)$ is the Euler characteristic of $M$, $\lvert \mathring{B} \rvert^2$ is the square of the trace-free second fundamental form of $M$ in $\mathbb{H}^3$, and $\mathcal{B}(M)$ is the bending of $M$.
\end{corollary}
\textit{Proof.} Consider a surface $M$ which approaches $\partial_\infty \mathbb{H}^3$ at a right angle. Hence, we can apply Theorem \ref{t2} to obtain
$$\mathcal{A}_R(M)=\int_M H^2\,dA-\int_M \overline{H}^2\,d\overline{A}\,,$$
where $\overline{H}$ is the mean curvature of $M$ as a surface in $\mathbb{R}^3$ and $d\overline{A}$ is the Euclidean area element on $M$. 

Additionally, applying Theorem 1 of \cite{C} (see \eqref{equality} as well) to the second integral above, we have
\begin{equation}\label{use}
	\mathcal{A}_R(M)=\int_M H^2\,dA-\int_M \overline{H}^2\,d\overline{A}=-\int_M \overline{\lambda}\,d\overline{A}+\int_MR \,dA\,,
\end{equation}
where $R$ is the extrinsic scalar curvature of $M$ as a surface in $\mathbb{H}^3$ and $dA$ is the hyperbolic area element on $M$. Observe that the hyperbolic bending cancels out with the corresponding quantity of \eqref{equality}. For the first integral on the right-hand side of \eqref{use}, we use the (Euclidean) Gauss-Bonnet formula \eqref{GB} to get
\begin{equation}\label{use2}
	\int_M\overline{\lambda}\,d\overline{A}=\int_M \overline{K}\,d\overline{A}=2\pi\chi(M)-\oint_{\partial M}\overline{\kappa}_g\,d\overline{s}=2\pi\chi(M)\,,
\end{equation}
since $\overline{\lambda}=\overline{K}$ coincides with the Gaussian curvature and $\partial M$ is a geodesic of $M$ due to the orthogonality of the intersection between $M$ and $\partial_\infty\mathbb{H}^3=\{z=0\}$.

Finally, solving for $R$ in \eqref{ChenB} and substituting it and \eqref{use2} in \eqref{use}, we end up with
$$\mathcal{A}_R(M)=-2\pi\chi(M)-\frac{1}{2}\int_M \lvert \mathring{B}\rvert^2\,dA+\int_M H^2\,dA\,.$$
This finishes the proof. \hfill$\square$
\\

The Gauss-Bonnet type formula \eqref{GBT3} gives an expression of the renormalized area $\mathcal{A}_R$ of a surface $M$ decomposed into three terms: a topological invariant term involving the Euler characteristic, a conformal invariant term which is precisely the integral of Chen's quantity (see Remark \ref{rem}), and an extrinsic non-conformally invariant term which is the bending of the surface. Of course, if the surface $M$ is minimal, its bending is identically zero and, hence, \eqref{GBT3} reduces to Alexakis-Mazzeo's expression of Theorem 1.2 of \cite{AM}.

\subsection{Hypersurfaces in $\mathbb{H}^5$}

We next consider the case $n=2$, that is, $M$ is a hypersurface in $\mathbb{H}^5$. When $M$ is also assumed to be minimal in $\mathbb{H}^5$, a Gauss-Bonnet type formula for its renormalized area was presented in Theorem 1.1 of \cite{T}.

In the following result we will show a relation between the renormalized area $\mathcal{A}_R$ of an asymptotically minimal hypersurface $M$ in $\mathbb{H}^5$ (regardless of $M$ being minimal or not) and its Euclidean and hyperbolic bendings.

\begin{theorem}\label{t3} Let $M$ be an asymptotically minimal (of order $2$) hypersurface in $\mathbb{H}^5$. Then, the renormalized area of $M$ is given by
	\begin{eqnarray}\label{AR5}
		\mathcal{A}_R(M)&=&\int_M H^4\,dA-\int_M \overline{H}^4\,d\overline{A}-2\left(\int_M H^2 \lambda\,dA-\int_M\overline{H}^2\overline{\lambda}\,d\overline{A}\right)\nonumber\\
		&&+\frac{1}{12}\left(\int_M \lvert E\rvert^2\,dA-\int_M \lvert\overline{E}\rvert^2\,d\overline{A}\right)-\frac{1}{12}\int_M \left(2\lvert\mathring{B}\rvert^2+\Delta \lvert\mathring{B}\rvert^2\right)dA\,,
	\end{eqnarray}
where $H$ and $\lambda$ are the mean and the scalar curvatures, $\lvert E\rvert^2$ is the square of the trace-free Ricci tensor (the Euclidean analogues are denoted with an upper bar), and $\lvert\mathring{B}\rvert^2$ is the square of the trace-free second fundamental form of $M$ in $\mathbb{H}^5$. Moreover, all the integrals of \eqref{AR5} are convergent.
\end{theorem}
\textit{Proof.} Let $M$ be a hypersurface in $\mathbb{H}^5$ with mean curvature $H=\mathcal{O}(z^2)$. For a sufficiently small value $\epsilon>0$ fixed, the part $M_\epsilon=M\cap\{z\geq \epsilon\}$ of $M$ fully lies in $\mathbb{H}^5$. It then follows from \eqref{formulaarea} (recall that $n=2$) that
\begin{eqnarray}\label{long}
	\mathcal{A}(M_\epsilon)&=&\int_{M_\epsilon} \overline{H}^4\,d\overline{A}-\int_{M_\epsilon}H^4 \,dA-2\left(\int_{M_\epsilon} \overline{H}^2\overline{R}\,d\overline{A}-\int_{M_\epsilon} H^2 R\,dA\right)\nonumber\\
	&&+\left(\int_{M_\epsilon} \overline{\lambda}^2\,d\overline{A}-\int_{M_\epsilon}\lambda^2\,dA\right)-2\int_{M_\epsilon}\lambda\,dA\,.
\end{eqnarray}
On the one hand, the last integral in \eqref{long} can be rewritten from \eqref{extrinsic-intrinsic} and \eqref{ChenB} as
\begin{equation}\label{long2}
	-2\int_{M_\epsilon}\lambda\,dA=-2\int_{M_\epsilon}\left(R-1\right)dA=-2\int_{M_\epsilon}H^2\,dA+\frac{1}{6}\int_{M_\epsilon}\lvert\mathring{B}\rvert^2\,dA+2\mathcal{A}(M_\epsilon)\,.
\end{equation}
On the other hand, the difference between the Euclidean and hyperbolic total squared scalar curvatures can be simplified employing the Chern-Gauss-Bonnet formula \eqref{CGB}. More precisely,
\begin{eqnarray*}
	\int_{M_\epsilon}\overline{\lambda}^2\,d\overline{A}-\int_{M_\epsilon}\lambda^2\,dA&=&\left(\frac{4}{3}\pi^2\chi(M_\epsilon)-\frac{1}{24}\int_{M_\epsilon}\lvert \overline{W}\rvert^2\,d\overline{A}+\frac{1}{12}\int_{M_\epsilon}\lvert\overline{E}\rvert^2\,d\overline{A}-\frac{1}{3}\oint_{\partial M_\epsilon}\overline{S}\,d\overline{s}\right)\nonumber\\
	&&-\left(\frac{4}{3}\pi^2\chi(M_\epsilon)-\frac{1}{24}\int_{M_\epsilon}\lvert W\rvert^2\,dA+\frac{1}{12}\int_{M_\epsilon}\lvert E\rvert^2\,dA-\frac{1}{3}\oint_{\partial M_\epsilon}S\,ds\right)\nonumber\\
	&=&\frac{1}{12}\left(\int_{M_\epsilon}\lvert \overline{E}\rvert^2\,d\overline{A}-\int_{M_\epsilon}\lvert E\rvert^2\,dA\right)-\frac{1}{3}\left(\oint_{\partial M_\epsilon}\overline{S}\,d\overline{s}-\oint_{\partial M_\epsilon}S\,ds\right).
\end{eqnarray*}
Observe that in the last equality we have canceled out the terms with the Euler characteristic, since the topology of $M_\epsilon$ does not change by changing the metric. In addition, when conformally changing the metric according to $\overline{g}=z^2g$, the Weyl tensor changes pointwise as $\overline{W}=z^{-2} W$, hence the integrals involving $W$ and $\overline{W}$ also cancel out. In conclusion, combining \eqref{long} and \eqref{long2} and the above formula we see that
\begin{eqnarray*}
	\mathcal{A}(M_\epsilon)&=&\int_{M_\epsilon} H^4\,dA-\int_{M_\epsilon}\overline{H}^4\,d\overline{A}-2\left(\int_{M_\epsilon} H^2R\,dA-\int_{M_\epsilon}\overline{H}^2\overline{R}\,d\overline{A}\right)+2\int_{M_\epsilon}H^2\,dA\nonumber\\
	&&+\frac{1}{12}\left(\int_{M_\epsilon}\lvert E\rvert^2\,dA-\int_{M_\epsilon}\lvert\overline{E}\rvert^2\,d\overline{A}\right)-\frac{1}{3}\left(\oint_{\partial M_\epsilon}S\,ds-\oint_{\partial M_\epsilon}\overline{S}\,d\overline{s}\right)-\frac{1}{6}\int_{M_\epsilon}\lvert\mathring{B}\rvert^2\,dA\,.
\end{eqnarray*}
Since $H=\mathcal{O}(z^2)$, it follows from \eqref{HRhyp-euc} that $\overline{\xi}_z=0$ at $z=0$, that is, $M$ approaches the ideal boundary $\partial_\infty\mathbb{H}^5$ at a right angle. Thus, $\partial M$ is totally geodesic in $M$ (this assertion follows directly from the orthogonality condition and the definition of the second fundamental form, after employing Gauss' formula), and, hence, $\overline{S}=0$ at $z=0$ (c.f., \eqref{Scurvature}). We then deduce that $\overline{S}=\mathcal{O}(\epsilon)$ at $\partial M_\epsilon$ and so its integral does not contribute to the renormalized area $\mathcal{A}_R$ of $M$ since it has no constant term when expanding it as a combination of powers of $\epsilon>0$. On the other hand, Lemma \ref{Claim} shows that the other boundary integral has no constant term when expanding it in terms of $\epsilon>0$ either.

Finally, adapting the proof of Theorem 3.2 of \cite{T} (see Lemma \ref{laplacian}, for details), we deduce that the constant term of
$$\int_{M_\epsilon} 2\lvert \mathring{B}\rvert^2\,dA\,,$$
is the same as that of
$$\int_{M_\epsilon}\left(2\lvert \mathring{B}\rvert^2+\Delta \lvert\mathring{B}\rvert^2\right)dA\,,$$
with the advantage that this last integral is convergent as $\epsilon\to 0^+$. Thus, we substitute the integrals in the above expression for $\mathcal{A}(M_\epsilon)$. The first part of the proof then follows by noticing that from \eqref{extrinsic-intrinsic}, $\lambda=R-1$ holds and so we can combine the third and fifth integrals in that expression.

We now focus our attention on the convergence of the integrals arising in \eqref{AR5}. Clearly, all the Euclidean integrals are convergent. In addition, the last integral of \eqref{AR5} is also convergent as will be shown in Lemma \ref{laplacian}. Moreover, it follows from \eqref{HRhyp-euc} and $dA=z^{-4}d\overline{A}$ that
\begin{eqnarray*}
	\int_M H^4\,dA&=&\int_M\left(z\overline{H}+\overline{\xi}_z\right)^4\frac{1}{z^4}\,d\overline{A}\\
	&=&\int_M \overline{H}^4\,d\overline{A}+4\int_M \overline{H}\,\frac{\overline{\xi}_z^3}{z^3}\,d\overline{A}+6\int_M\overline{H}^2\,\frac{\overline{\xi}_z^2}{z^2}\,d\overline{A}+4\int_M\overline{H}^3\,\frac{\overline{\xi}_z}{z}\,d\overline{A}+\int_M\frac{\overline{\xi}_z^4}{z^4}\,d\overline{A}\,,
\end{eqnarray*}
where all the integrals converge using \emph{only} that $\overline{\xi}_z=\mathcal{O}(z)$. Similarly,
\begin{eqnarray*}
	\int_M H^2R\,dA&=&\int_M \overline{H}^2\overline{R}\,d\overline{A}+2\int_M \overline{H}^3\,\frac{\overline{\xi}_z}{z}\,d\overline{A}+5\int_M\overline{H}^2\,\frac{\overline{\xi}_z^2}{z^2}\,d\overline{A}+2\int_M\overline{H}\,\overline{R}\,\frac{\overline{\xi}_z}{z}\,d\overline{A} \\
	&&+4\int_M\overline{H}\,\frac{\overline{\xi}_z^3}{z^3}\,d\overline{A}+\int_M \overline{R}\,\frac{\overline{\xi}_z^2}{z^2}\,d\overline{A}+\int_M\frac{\overline{\xi}_z^4}{z^4}\,d\overline{A}<\infty\,,
\end{eqnarray*}
\emph{only} from $\overline{\xi}_z=\mathcal{O}(z)$.

The convergence of the total squared mean curvature integral follows from the asymptotically minimal of order $2$ condition (see also Remark \ref{cond} below). Hence, it only remains to check the convergence of the integral of the square of the trace-free Ricci tensor. To see this, observe that $\lvert E\rvert^2$ can be rewritten as
\begin{equation}\label{Ricci}
	\lvert E\rvert^2=\lvert\mathring{B}^2\rvert^2-4H{\rm Trace}(\mathring{B}^3)+4H^2\lvert\mathring{B}\rvert^2-\frac{1}{4}\lvert\mathring{B}\rvert^4\,.
\end{equation}
This expression follows from the definition of the Ricci tensor and the Gauss' equation. To deduce it in a simple way, it is useful to employ \eqref{RinB}, \eqref{extrinsic-intrinsic}, and the relation between the second fundamental form $B$ and its trace-free counterpart $\mathring{B}$. It then follows from \eqref{HRhyp-euc} that each term in \eqref{Ricci} are $\mathcal{O}(z^4)$ and so the integral of $\lvert E\rvert^2$ converges. \hfill$\square$

\begin{rem}\label{cond} Recall that a hypersurface $M$ in $\mathbb{H}^5$ is asymptotically minimal (of order $2$) if $H=\mathcal{O}(z^2)$. In particular, we obtain from \eqref{HRhyp-euc} that $\overline{\xi}_z=\mathcal{O}(z)$. Hence, necessarily $M$ must meet the ideal boundary $\partial_\infty\mathbb{H}^5$ at a right angle. However, the asymptotically minimal condition in this case is stronger than this orthogonality condition (contrary to the case of surfaces). 
	
We point out here that most of the steps on the proof of Theorem \ref{t3} will follow exactly the same if we just assume that $M$ meets $\partial_\infty\mathbb{H}^5$ at a right angle. We are only using that $H=\mathcal{O}(z^2)$ to show the convergence of the total squared mean curvature and to apply Lemma \ref{laplacian}. More precisely, from $dA=z^{-4}d\overline{A}$, it follows that $H=\mathcal{O}(z^2)$ is equivalent to 
$$\int_M H^2\,dA<\infty\,.$$
\end{rem}

We next use the general relation \eqref{AR5}, together with the Chern-Gauss-Bonnet formula \eqref{CGB} and Chen's result \cite[Theorem 1]{C}, to deduce an extension of the formula of Theorem 1.1 of \cite{T} for the renormalized area of minimal hypersurfaces in $\mathbb{H}^5$ to the more general case.

\begin{corollary}\label{c2} Let $M$ be an asymptotically minimal (of order $2$) hypersurface in $\mathbb{H}^5$. Then, the renormalized area of $M$ satisfies
	\begin{eqnarray}\label{GBT5}
		\mathcal{A}_R(M)&=&\frac{4}{3}\pi^2\chi(M)-\frac{1}{144}\int_M\lvert \mathring{B}\rvert^4\,dA-\frac{1}{24}\int_M\lvert W\rvert^2\,dA+\frac{1}{12}\int_M\lvert E\rvert^2\,dA\nonumber\\
		&&+\mathcal{B}(M)-2\int_M H^2\lambda\,dA-\frac{1}{12}\int_M \left(2\lvert\mathring{B}\rvert^2+\Delta\lvert\mathring{B}\rvert^2\right) dA \,,
	\end{eqnarray}
	where $\chi(M)$ is the Euler characteristic, $\mathring{B}$ is the trace-free second fundamental form, $W$ is the Weyl tensor, $\lvert E\rvert^2$ is the square of the trace-free Ricci tensor, $\mathcal{B}(M)$ represents the bending of $M$, and $H$ and $\lambda$ are the mean and scalar curvatures of $M$ in $\mathbb{H}^5$. 
\end{corollary} 
\textit{Proof.} Assume $M$ is an asymptotically minimal hypersurface in $\mathbb{H}^5$. From Theorem \ref{t3} we get (after suitably rearranging)
\begin{eqnarray}
	\mathcal{A}_R(M)&=&\int_M H^4\,dA-2\int_MH^2 \lambda\,dA+\frac{1}{12}\int_M \lvert E\rvert^2\,dA-\frac{1}{12}\int_M\left(2\lvert\mathring{B}\rvert^2+\Delta\lvert\mathring{B}\rvert^2\right)dA\nonumber\\
	&&-\int_M \left(\overline{H}^2-\overline{R}\right)^2\,d\overline{A}+\int_M\overline{\lambda}^2\,d\overline{A}-\frac{1}{12}\int_M \lvert\overline{E}\rvert^2\,d\overline{A}\,,\label{help}
\end{eqnarray}
where we have used that $\overline{\lambda}=\overline{R}$ holds from \eqref{extrinsic-intrinsic}.

The first integral on the second line of \eqref{help} can be rewritten employing the relation \eqref{ChenB} and Theorem 1 of \cite{C} (see also \eqref{equality} above). Indeed, we have
$$\int_M \left(\overline{H}^2-\overline{R}\right)^2\,d\overline{A}=\int_M\left(H^2-R\right)^2\,dA=\frac{1}{144}\int_M \lvert\mathring{B}\rvert^4\,dA\,.$$
Moreover, for the last two integrals in \eqref{help} we use the Chern-Gauss-Bonnet formula \eqref{CGB} to rewrite them, obtaining
$$\int_M \overline{\lambda}^2\,d\overline{A}-\frac{1}{12}\int_M \lvert\overline{E}\rvert^2\,d\overline{A}=\frac{4}{3}\pi^2\chi(M)-\frac{1}{24}\int_M\lvert \overline{W}\rvert^2\,d\overline{A}-\frac{1}{3}\oint_{\partial M} \overline{S}\,d\overline{s}\,.$$
Observe that the boundary integral vanishes since $M$ approaches $\partial_\infty \mathbb{H}^5$ at a right angle (c.f., proof of Theorem \ref{t3}).

Combining everything and substituting it in \eqref{help} we conclude with
\begin{eqnarray*}
	\mathcal{A}_R(M)&=&\int_M H^4\,dA-2\int_M H^2\lambda\,dA+\frac{1}{12}\int_M \lvert E\rvert^2\,dA-\frac{1}{12}\int_M\left(2\lvert\mathring{B}\rvert^2+\Delta \lvert\mathring{B}\rvert^2\right)dA\\&&-\frac{1}{144}\int_M\lvert\mathring{B}\rvert^4\,dA+\frac{4}{3}\pi^2\chi(M)-\frac{1}{24}\int_M \lvert \overline{W}\rvert^2\,d\overline{A}\,.
\end{eqnarray*}
The result then follows since the Weyl tensor satisfies $\lvert\overline{W}\rvert^2d\overline{A}=\lvert W\rvert^2dA$. \hfill$\square$
\\

Observe that the Gauss-Bonnet type formula \eqref{GBT5} for a hypersurface in $\mathbb{H}^5$ is decomposed in three different parts: a topological part involving the Euler characteristic, a conformal invariant part which is composed by a multiple of the integral of Chen's quantity and another term related to the Weyl tensor, and an extrinsic non-conformally invariant part involving the curvatures of the hypersurface in $\mathbb{H}^5$. This decomposition is analogous to that of the lower dimensional case $n=1$ (see Corollary \ref{c1}).

Assume that $M$ is a minimal hypersurface in $\mathbb{H}^5$. Then, $H=0$ holds identically and, in particular, $M$ is asymptotically minimal. In addition, the bending of $M$ vanishes. Similarly, the second to last integral in \eqref{GBT5} is also zero. Therefore, we obtain that the renormalized area of a minimal hypersurface in $\mathbb{H}^5$ satisfies
\begin{eqnarray*}
		\mathcal{A}_R(M)&=&\frac{4}{3}\pi^2\chi(M)-\frac{1}{144}\int_M \lvert\mathring{B}\rvert^4\,dA-\frac{1}{24}\int_M\lvert W\rvert^2\,dA\\
		&&+\frac{1}{12}\int_M\lvert E\rvert^2\,dA-\frac{1}{12}\int_M\left(2\lvert\mathring{B}\rvert^2+\Delta\lvert \mathring{B}\rvert^2\right)dA\,.
\end{eqnarray*}
This expression was first obtained in Theorem 1.1 of \cite{T}.

\section{Hypersurfaces in Poincar\'e-Einstein Manifolds}

In this section we extend the Gauss-Bonnet type formulas given in Corollaries \ref{c1} and \ref{c2} to the case where the ambient manifold is a Poincar\'e-Einstein space. We begin by briefly reviewing the definition of these spaces (for more details, we refer the reader to \cite{AM,FG,G,T}).

Let $X^{2n+1}$ be a $(2n+1)$-dimensional compact manifold with boundary $\partial X^{2n+1}$ endowed with a complete Riemannian metric $\widetilde{g}$ defined on the interior $\mathring{X}^{2n+1}$. In what follows, we will denote the interior of $X^{2n+1}$ by $\widetilde{N}^{2n+1}=\mathring{X}^{2n+1}$ and refer to $\partial X^{2n+1}$ as the asymptotic boundary.

A function $r:X^{2n+1}\longrightarrow\mathbb{R}$ is a defining function for the boundary $\partial X^{2n+1}$ if $r>0$ holds on the interior $\widetilde{N}^{2n+1}=\mathring{X}^{2n+1}$, $r=0$ holds on the boundary $\partial X^{2n+1}$, and $dr\neq 0$ on $\partial X^{2n+1}$. The manifold $\widetilde{N}^{2n+1}=\mathring{X}^{2n+1}$ together with the metric $\widetilde{g}$ is said to be conformally compact if $r^2\widetilde{g}$ extends as a (smooth and nondegenerate) metric to $X^{2n+1}$, where $r$ is a defining function for $\partial X^{2n+1}$. The sectional curvatures of $\widetilde{N}^{2n+1}$ tend to $-\lvert dr\rvert^2_{r^2\widetilde{g}}$ as the points approach the boundary $\partial X^{2n+1}$. In particular, if $\lvert dr\rvert^2_{r^2 \widetilde{g}}=1$ holds along $\partial X^{2n+1}$, the manifold $\widetilde{N}^{2n+1}$ with the metric $\widetilde{g}$ is referred to as asymptotically hyperbolic. If $\lvert dr\rvert^2_{r^2\widetilde{g}}=1$ holds in a neighborhood of $\partial X^{2n+1}$, rather than just at the asymptotic boundary itself, the function $r$ is said to be a special defining function for $\partial X^{2n+1}$. As shown by Graham \cite{G}, asymptotically hyperbolic manifolds admit special defining functions.

A conformally compact manifold $(\widetilde{N}^{2n+1},\widetilde{g}\,)$ which satisfies the Einstein condition
\begin{equation}\label{Einstein}
	\widetilde{{\rm Ric}}=-2n \widetilde{g}\,,
\end{equation}
is called a Poincar\'e-Einstein space. Here, $\widetilde{{\rm Ric}}$ denotes the Ricci tensor of $\widetilde{N}^{2n+1}$. These spaces are special cases of asymptotically hyperbolic manifolds. Particular examples of Poincar\'e-Einstein spaces are the hyperbolic spaces $\mathbb{H}^{2n+1}$. For these spaces the asymptotic boundary is the one-point compactification of what we defined as the ideal boundary $\partial_\infty\mathbb{H}^{2n+1}=\{z=0\}$ in previous sections.

From now on, we will assume that $(\widetilde{N}^{2n+1},\widetilde{g}\,)$ is a Poincar\'e-Einstein space. It then follows from the standard Ricci decomposition of the Riemann curvature $\widetilde{{\rm Rm}}$ of $\widetilde{N}^{2n+1}$ and the Einstein condition \eqref{Einstein} that for every $i,j,k,l=1,...,2n+1$,
\begin{eqnarray}\label{RmTilde}
\widetilde{{\rm Rm}}(e_i,e_j,e_k,e_l)&=&\widetilde{g}(e_i,e_l)\widetilde{g}(e_j,e_k)-\widetilde{g}(e_i,e_k)\widetilde{g}(e_j,e_l)+\widetilde{W}(e_i,e_j,e_k,e_l)\nonumber\\
&=&\delta_{il}\delta_{jk}-\delta_{ik}\delta_{jl}+\widetilde{W}(e_i,e_j,e_k,e_l)\,,
\end{eqnarray}
where $\delta$ represents the Kronecker delta, $\widetilde{W}$ is the Weyl tensor of $\widetilde{N}^{2n+1}$, and $\{e_1,...,e_{2n},e_{2n+1}\}$ is an orthonormal frame. (When considering a hypersurface $M$ of $\widetilde{N}^{2n+1}$ the orthonormal frame will be taken to be adapted to $M$, meaning that $e_{2n+1}=\xi$ will be the unit normal to $M$.)

Let $M$ be a smooth and compact hypersurface with boundary embedded in a Poincar\'e-Einstein space $\widetilde{N}^{2n+1}$. All the definitions and formulas of Section 2 extend directly to this setting (substituting $z$ by the defining function $r$, when necessary). The only difference is that the sectional curvature of the ambient space $\widetilde{N}^{2n+1}$ may be now nonconstant (to distinguish it from that of a space form, we denote it here by $\mu$). Indeed, it follows from \eqref{RmTilde} that for every $i,j=1,...,2n$ and $i\neq j$,
\begin{equation}\label{mu}
	\mu_{ij}=\mu(e_i,e_j)=\widetilde{{\rm Rm}}(e_i,e_j,e_i,e_j)=-1+\widetilde{W}(e_i,e_j,e_i,e_j)=-1+\widetilde{W}_{ij}\,,
\end{equation}
where we have introduced the simpler notation $\widetilde{W}_{ij}\equiv \widetilde{W}(e_i,e_j,e_i,e_j)$. From this restricted to the orthonormal frame of principal directions $\{v_1,...,v_{2n}\}$, the Gauss' equation, and the definition \eqref{HandR} of the extrinsic scalar curvature, we obtain
\begin{eqnarray*}
	R&=&\frac{1}{2n(2n-1)}\sum_{i\neq j}\kappa_i\kappa_j=\frac{1}{2n(2n-1)}\sum_{i\neq j} \left(K(v_i,v_j)-\mu_{ij}\right)\\
	&=&\frac{1}{2n(2n-1)}\left(\sum_{i\neq j} K(v_i,v_j)-\sum_{i\neq j}(-1+\widetilde{W}_{ij})\right)\\
	&=&\lambda+1-\frac{1}{2n(2n-1)}\sum_{i\neq j} \widetilde{W}_{ij}\,,
\end{eqnarray*}
where we have used the definition \eqref{lambda} of the (intrinsic) scalar curvature. For simplicity, we define
\begin{equation}\label{W*}
	\widetilde{W}_*=\frac{1}{2n(2n-1)}\sum_{i\neq j} \widetilde{W}_{ij}=\frac{1}{2n(2n-1)}\sum_{i=1}^{2n}\sum_{j=1}^{2n} \widetilde{W}_{ij}\,,
\end{equation}
where the last equality follows since $\widetilde{W}_{ii}=0$ for every $i=1,...,2n+1$ (c.f., \eqref{RmTilde}). Observe that since the Weyl tensor $\widetilde{W}$ is trace-free, the quantity $\widetilde{W}_*$ is zero. Indeed, for every $i=1,...,2n+1$ fixed, we have
$$0={\rm Trace}(\widetilde{W})=\sum_{j=1}^{2n+1}\widetilde{W}_{ij}=\sum_{j=1}^{2n}\widetilde{W}_{ij}+\widetilde{W}(e_i,\xi,e_i,\xi)\,,$$
since $e_{2n+1}=\xi$. Adding this from $i=1$ to $i=2n$ and taking into account that $\widetilde{W}_{ij}=\widetilde{W}_{ji}$ (see \eqref{RmTilde} again), we then get
\begin{eqnarray*}
	2n(2n-1)\widetilde{W}_*&=&\sum_{i=1}^{2n}\sum_{j=1}^{2n}\widetilde{W}_{ij}=-\sum_{i=1}^{2n}\widetilde{W}(e_i,\xi,e_i,\xi)=-\sum_{i=1}^{2n+1}\widetilde{W}(e_i,\xi,e_i,\xi)+\widetilde{W}(\xi,\xi,\xi,\xi)\\&=&-\sum_{i=1}^{2n+1}\widetilde{W}(\xi,e_i,\xi,e_i)=-{\rm Trace}(\widetilde{W})=0\,.
\end{eqnarray*} 
Hence, \eqref{extrinsic-intrinsic} reads in our case
\begin{equation}\label{extrinsic-intrinsic-general}
	R=\lambda+1\,.
\end{equation}

Assume that the boundary $\partial M$ lies in the asymptotic boundary $\partial X^{2n+1}$ and that the mean curvature $H$ of the hypersurface $M$ satisfies $H=\mathcal{O}(r^{n})$. We refer to this property simply saying that $M$ is asymptotically minimal (of order $n$). To understand the renormalized area $\mathcal{A}_R$ of $M$, we consider the truncated hypersurfaces $M_\epsilon=M\cap\{r\geq \epsilon\}$, for $\epsilon>0$ sufficiently small. For these hypersurfaces we use \eqref{extrinsic-intrinsic-general} and expand
\begin{equation*}
	R^n=\left(\lambda+1\right)^n=1+\sum_{i=1}^{n}\begin{pmatrix} n \\ i \end{pmatrix} \lambda^i\,.
\end{equation*}
Integrating over $M_\epsilon$ and rearranging, we obtain
\begin{equation}\label{expand}
	\mathcal{A}(M_\epsilon)=\int_{M_\epsilon} R^n\,d\widetilde{A}-\sum_{i=1}^n\begin{pmatrix} n \\ i \end{pmatrix} \int_{M_\epsilon}\lambda^i\,d\widetilde{A}\,,
\end{equation}
where $d\widetilde{A}$ is the volume element on $M$ induced from the metric $\widetilde{g}$ of the Poincar\'e-Einstein ambient space $\widetilde{N}^{2n+1}$. The approach now will consist on adding and subtracting the necessary terms to the right-hand side of \eqref{expand} until the integral of Chen's conformal invariant quantity $\left(H^2-R\right)^n\widetilde{g}$ arises. This gives
$$\mathcal{A}(M_\epsilon)=(-1)^n\int_{M_\epsilon}\left(H^2-R\right)^n\,d\widetilde{A}-\sum_{i=0}^{n-1}(-1)^{n+i}\begin{pmatrix} n \\ i \end{pmatrix}\int_{M_\epsilon} H^{2(n-i)}R^i\,d\widetilde{A}-\sum_{i=1}^n\begin{pmatrix} n \\ i \end{pmatrix} \int_{M_\epsilon}\lambda^i\,d\widetilde{A}\,.$$
Hence, employing \eqref{ChenB} and considering the constant term of the expansion of $\mathcal{A}(M_\epsilon)$ in terms of $\epsilon>0$, we deduce that
\begin{eqnarray}\label{renormalizedgeneral}
	\mathcal{A}_R(M)&=&\left(\frac{-1}{2n(2n-1)}\right)^n\int_M \lvert\mathring{B}\rvert^{2n}\,d\widetilde{A}-\sum_{i=0}^{n-1}(-1)^{n+i}\begin{pmatrix} n \\ i\end{pmatrix}\int_M H^{2(n-i)}R^i\,d\widetilde{A}\nonumber\\
	&&-\sum_{i=1}^n\begin{pmatrix} n\\i\end{pmatrix} {\rm f.p.}\int_M \lambda^i\,d\widetilde{A}\,.
\end{eqnarray}
The integrals above converge, but perhaps for those with ${\rm f.p.}$ in front. We must rewrite these potentially divergent integrals to properly understand the renormalized area.

We next show all the details for the cases $n=1$ and $n=2$ separately. In these cases, we will employ the Chern-Gauss-Bonnet formula to rewrite the integral of $\lambda^n$. 

\begin{proposition}\label{last1}
	Let $M$ be a surface in a $3$-dimensional Poincar\'e-Einstein space $\widetilde{N}^3$ approaching the asymptotic boundary at a right angle. Then, the renormalized area of $M$ satisfies
	\begin{equation}\label{RA3}
		\mathcal{A}_R(M)=-2\pi\chi(M)-\frac{1}{2}\int_M \lvert \mathring{B}\rvert^2\,d\widetilde{A}+\widetilde{\mathcal{B}}(M)\,,
	\end{equation}
where $\chi(M)$ is the Euler characteristic of $M$, $\lvert\mathring{B}\rvert^2$ is the square of the trace-free second fundamental form of $M$ in $\widetilde{N}^3$, and $\widetilde{\mathcal{B}}(M)$ is the bending of $M$ in $\widetilde{N}^3$.
\end{proposition}
\textit{Proof.} Let $\widetilde{N}^3$ be a Poincar\'e-Einstein space and consider a surface $M$ approaching the asymptotic boundary at a right angle. For $\epsilon>0$ sufficiently small we denote by $M_\epsilon=M\cap\{r\geq \epsilon\}$ the truncated surfaces.

For $n=1$, equation \eqref{expand} reduces to
$$\mathcal{A}(M_\epsilon)=\int_{M_\epsilon} R\,d\widetilde{A}-\int_{M_\epsilon}\lambda\,d\widetilde{A}\,.$$
Adding and subtracting the bending $\widetilde{\mathcal{B}}(M)$ of $M$ in $\widetilde{N}^3$ and applying the Gauss-Bonnet formula \eqref{GB} to the last integral above (recall that for surfaces, $\lambda=K$ is the Gaussian curvature), we obtain
\begin{eqnarray*}
	\mathcal{A}(M_\epsilon)&=&\int_{M_\epsilon}H^2\,d\widetilde{A}-\int_{M_\epsilon}\left(H^2-R\right)d\widetilde{A}-2\pi\chi(M_\epsilon)+\oint_{\partial M_\epsilon} \kappa_g\,d\widetilde{s}\\
	&=&-2\pi\chi(M_\epsilon)-\frac{1}{2}\int_{M_\epsilon}\lvert \mathring{B}\rvert^2\,d\widetilde{A}+\int_{M_\epsilon} H^2\,d\widetilde{A}+\oint_{\partial M_\epsilon}\kappa_g\,d\widetilde{s}\,,
\end{eqnarray*}
where in the last equality we have used \eqref{ChenB}.

Finally, we deduce from (3.11) of \cite{AM} that the total geodesic curvature term above has no constant term when considering its expansion in terms of $\epsilon>0$. This concludes the proof. \hfill$\square$
\\

The above result and proof are essentially the same as those of Proposition 3.1 of \cite{AM}. We have opted to include them here since they serve as an illustration for the general process and, in particular, for the following higher dimensional case $n=2$.

\begin{proposition}\label{last2} Let $M$ be an asymptotically minimal (of order $2$) hypersurface in a $5$-dimensional Poincar\'e-Einstein space $\widetilde{N}^5$. Then, the renormalized area of $M$ satisfies
	\begin{eqnarray}\label{RA5}
		\mathcal{A}_R(M)&=&\frac{4}{3}\pi^2\chi(M)-\frac{1}{144}\int_M\lvert \mathring{B}\rvert^4\,d\widetilde{A}-\frac{1}{24}\int_M\lvert W\rvert^2\,d\widetilde{A}+\frac{1}{12}\int_M\lvert E\rvert^2\,d\widetilde{A}\nonumber\\
		&&+\widetilde{\mathcal{B}}(M)-2\int_MH^2\lambda\,d\widetilde{A}-\frac{1}{12}\int_M \left(2\lvert\mathring{B}\rvert^2+\Delta \lvert\mathring{B}\rvert^2\right)d\widetilde{A}\,,
	\end{eqnarray}
	where $\chi(M)$ is the Euler characteristic, $\mathring{B}$ is the trace-free second fundamental form, $W$ is the Weyl tensor, $\lvert E\rvert^2$ is the square of the trace-free Ricci tensor, $\widetilde{\mathcal{B}}(M)$ represents the bending of $M$, and $H$ and $\lambda$ are the mean and scalar curvatures of $M$ in $\widetilde{N}^5$. 
\end{proposition} 
\textit{Proof.} Let $M$ be a hypersurface in $\widetilde{N}^5$ satisfying the conditions of the statement and denote by $M_\epsilon=M\cap\{r\geq \epsilon\}$, $\epsilon>0$ sufficiently small, the truncated hypersurfaces. Since $n=2$ holds, \eqref{expand} reads
	$$\mathcal{A}(M_\epsilon)=\int_{M_\epsilon} R^2\,d\widetilde{A}-2\int_{M_\epsilon}\lambda\,d\widetilde{A}-\int_{M_\epsilon}\lambda^2\,d\widetilde{A}\,.$$
For the first integral, we add and subtract the necessary terms until the integral of Chen's conformal invariant quantity appears. That is,
\begin{eqnarray*}
	\int_{M_\epsilon}R^2\,d\widetilde{A}&=&\int_{M_\epsilon} \left(H^2-R\right)^2\,d\widetilde{A}-\int_{M_\epsilon} H^4\,d\widetilde{A}+2\int_{M_\epsilon} H^2R\,d\widetilde{A}\\
	&=&\frac{1}{144}\int_{M_\epsilon} \lvert\mathring{B}\rvert^4\,d\widetilde{A}-\widetilde{\mathcal{B}}(M_\epsilon)+2\int_{M_\epsilon} H^2R\,d\widetilde{A}\,,
\end{eqnarray*}
where we have employed the relation \eqref{ChenB} in the second equality.

Now, we rewrite the second integral in the expression of the area, employing \eqref{extrinsic-intrinsic-general} and \eqref{ChenB}, as
\begin{eqnarray*}
	-2\int_{M_\epsilon}\lambda\,d\widetilde{A}&=&-2\int_{M_\epsilon} R\,d\widetilde{A}+2\mathcal{A}(M_\epsilon)=-2\int_{M_\epsilon}H^2\,d\widetilde{A}+\frac{1}{6}\int_{M_\epsilon}\lvert\mathring{B}\rvert^2\,d\widetilde{A}+2\mathcal{A}(M_\epsilon)\\
	&=&-2\int_{M_\epsilon}H^2\,d\widetilde{A}+\frac{1}{12}\int_{M_\epsilon}\left(2\lvert\mathring{B}\rvert^2+\Delta \lvert\mathring{B}\rvert^2\right)d\widetilde{A}+2\mathcal{A}(M_\epsilon)\,.
\end{eqnarray*}
The third equality follows in analogy with the proof of Theorem \ref{t3} (more precisely, see Lemma \ref{laplacian} below).

We next employ the Chern-Gauss-Bonnet formula \eqref{CGB} to obtain
$$-\int_{M_\epsilon}\lambda^2\,d\widetilde{A}=-\frac{4}{3}\pi^2\chi(M_\epsilon)+\frac{1}{24}\int_{M_\epsilon}\lvert W\rvert^2\,d\widetilde{A}-\frac{1}{12}\int_{M_\epsilon}\lvert E\rvert^2\,d\widetilde{A}+\frac{1}{3}\oint_{\partial M_\epsilon}S\,d\widetilde{s}\,.$$
Combining everything and rearranging, we get
\begin{eqnarray*}
	\mathcal{A}(M_\epsilon)&=&\frac{4}{3}\pi^2\chi(M_\epsilon)-\frac{1}{144}\int_{M_\epsilon}\lvert\mathring{B}\rvert^4\,d\widetilde{A}-\frac{1}{24}\int_{M_\epsilon}\lvert W\rvert^2\,d\widetilde{A}+\frac{1}{12}\int_{M_\epsilon}\lvert E\rvert^2\,d\widetilde{A}\\
	&&+\widetilde{\mathcal{B}}(M_\epsilon)-2\int_{M_\epsilon}H^2R\,d\widetilde{A}+2\int_{M_\epsilon}H^2\,d\widetilde{A}-\frac{1}{12}\int_{M_\epsilon}\left(2\lvert\mathring{B}\rvert^2+\Delta\lvert\mathring{B}\rvert^2\right)d\widetilde{A}\\
	&&-\frac{1}{3}\oint_{\partial M_\epsilon}S\,d\widetilde{s}\,.
\end{eqnarray*}
Observe that we can combine the second and third integrals in the second line employing \eqref{extrinsic-intrinsic-general} as
$$-2\int_{M_\epsilon} H^2R\,d\widetilde{A}+2\int_{M_\epsilon}H^2\,d\widetilde{A}=-2\int_{M_\epsilon}H^2(R-1)\,d\widetilde{A}=-2\int_{M_\epsilon}H^2\lambda\,d\widetilde{A}\,.$$
As explained in the proof of Theorem \ref{t3} (see also Lemma \ref{Claim}), the boundary integral in the above formula will not contribute to the renormalized area. This finishes the proof. \hfill$\square$

\section{Submanifolds of Arbitrary Codimension}

In this final section we will briefly discuss the extensions of our results for even-dimensional submanifolds of arbitrary codimension in Poincar\'e-Einstein spaces. We will assume that $M$ is a $2n$-dimensional submanifold of a Poincar\'e-Einstein $m$-dimensional space $(\widetilde{N}^m,\widetilde{g}\,)$, $m\geq 2n+1$. Let $\xi_1,...,\xi_{m-2n}$ be $m-2n$ mutually orthogonal unit normal vectors to $M$ with respect to $\widetilde{g}$. For each $\alpha=1,...,m-2n$, we define $H_\alpha={\rm Trace}B_\alpha/(2n)$, where $B_\alpha$ represents the second fundamental form associated to the unit normal $\xi_\alpha$. This gives us the mean curvature vector
$$H_T=\sum_{\alpha=1}^{m-2n} H_\alpha\,\xi_\alpha\,.$$
The extrinsic scalar curvature $R_T$ of $M$ can be considered to be
\begin{equation}\label{RT}
	R_T=\lvert H_T\rvert^2-\frac{1}{2n(2n-1)}\sum_{\alpha=1}^{m-2n}\lvert\mathring{B}_\alpha\rvert^2=\lvert H_T\rvert^2-\frac{1}{2n(2n-1)}\lvert\mathring{B}_T\rvert^2\,,
\end{equation}
where, for simplicity, we are denoting by $\lvert \mathring{B}_T\rvert^2$ the square of the total trace-free second fundamental form, that is, the sum of the squares of all the trace-free second fundamental forms $\mathring{B}_\alpha$, $\alpha=1,...,m-2n$. The expression \eqref{RT} follows from the definition of the extrinsic scalar curvature $R_T$ of $M$ given in (2.14) of \cite{C}, after noticing that for each $\alpha=1,...,m-2n$ fixed, \eqref{ChenB} holds (of course, here we must understand \eqref{ChenB} with the subindex $\alpha$ for each operator).

Finally, we observe that the Gauss' equation in combination with a computation analogue to that of \eqref{mu} yields
$$R_T=\lambda-\widetilde{W}_*+1\,,$$
where $\lambda$ is the intrinsic scalar curvature and $\widetilde{W}_*$ is the average of the components of the Weyl tensor as introduced in \eqref{W*}. We point out here that when the codimension of $M$ is greater than $1$, $\widetilde{W}_*$ may not be zero. Consequently, the analogue of \eqref{expand} now reads
\begin{equation}\label{expandhigh}
	\mathcal{A}(M_\epsilon)=\int_{M_\epsilon}R_T^n\,d\widetilde{A}-\sum_{i=1}^n\begin{pmatrix} n \\ i \end{pmatrix} \int_{M_\epsilon} \left(\lambda-\widetilde{W}_*\right)^i\,d\widetilde{A}\,.
\end{equation}

We are now in a position to prove the first results. The following is the extension of Proposition \ref{last1} to arbitrary codimension (c.f., Proposition 3.1 of \cite{AM}).

\begin{proposition}\label{last3} Let $M$ be a surface in a Poincar\'e-Einstein space $(\widetilde{N}^m,\widetilde{g}\,)$ and assume that $M$ approaches the asymptotic boundary at a right angle. Then, the renormalized area of $M$ satisfies
	$$\mathcal{A}_R(M)=-2\pi\chi(M)-\frac{1}{2}\int_M\lvert \mathring{B}_T\rvert^2\,d\widetilde{A}+\int_M \lvert H_T\rvert^2\,d\widetilde{A}+\int_M \widetilde{W}_{12}\,d\widetilde{A}\,,$$
where $\chi(M)$ is the Euler characteristic of $M$, $\lvert\mathring{B}_T\rvert^2$ is the square of the total trace-free second fundamental form of $M$ in $\widetilde{N}^m$, $\lvert H_T\rvert^2=\widetilde{g}(H_T,H_T)$ is the magnitude square of the mean curvature vector, and $\widetilde{W}_{12}$ is the Weyl tensor of $\widetilde{N}^m$ evaluated at any orthonormal frame for $TM$.
\end{proposition}
\textit{Proof.} The proof is analogous to that of Proposition \ref{last1}. However, since $\widetilde{W}_*=\widetilde{W}_{12}$ may now be different from zero, we need to keep that term, which comes from \eqref{expandhigh}. The convergence of the last integral follows from the pointwise transformation of the Weyl tensor for a conformal change of metric. \hfill$\square$
\\

Our next result will extend Theorem \ref{t2}.

\begin{proposition} Let $M$ be a surface in $\mathbb{H}^m$ approaching the ideal boundary $\partial_\infty\mathbb{H}^m$ at a right angle. Then, the renormalized area of $M$ is given by
	$$\mathcal{A}_R(M)=\int_M \lvert H_T\rvert^2\,dA-\int_M \lvert \overline{H}_T\rvert^2\,d\overline{A}\,,$$
	where $H_T$ is the mean curvature vector of $M$ as a surface in $\mathbb{H}^m$, $dA$ is the area element on $M$ induced from the hyperbolic metric, while $\overline{H}_T$ and $d\overline{A}$ are the analogues when regarding (after a conformal change of metric) $M$ as a surface in the Euclidean space $\mathbb{R}^m$.
\end{proposition}
\textit{Proof.} Let $M$ be a surface in $\mathbb{H}^m$. From Proposition \ref{last3}, it follows that
$$\mathcal{A}_R(M)=-2\pi\chi(M)-\frac{1}{2}\int_M \lvert\mathring{B}_T\rvert^2\,dA+\int_M \lvert H_T\rvert^2\,dA\,,$$
since $\widetilde{W}_{12}\equiv 0$ for the hyperbolic metric.

We now deduce from \eqref{RT} and Theorem 1 of \cite{C} that
$$\frac{1}{2}\int_M \lvert\mathring{B}_T\rvert^2\,dA=\int_M \left(\lvert H_T\rvert^2-R_T\right)dA=\int_M \left(\lvert \overline{H}_T\rvert^2-\overline{R}_T\right)d\overline{A}\,.$$
Observe that the (Euclidean) extrinsic scalar curvature $\overline{R}_T=\overline{\lambda}=\overline{K}$ is the Gaussian curvature and so we can use the classical Gauss-Bonnet formula \eqref{GB}.

Substituting everything in the formula for $\mathcal{A}_R$ and employing the orthogonality condition to get rid of the boundary integral, we prove the result. \hfill$\square$
\\

We now consider the analogue results for $4$-dimensional submanifolds. The following propositions will extend Proposition \ref{last2} and Theorem \ref{t3}, respectively.

\begin{proposition}\label{lastnew} Let $M$ be an asymptotically minimal (of order $2$) $4$-dimensional submanifold in a Poincar\'e-Einstein space $\widetilde{N}^m$ and assume $\widetilde{W}_*=\mathcal{O}(r^4)$. Then, the renormalized area of $M$ satisfies
	\begin{eqnarray*}
		\mathcal{A}_R(M)&=&\frac{4}{3}\pi^2\chi(M)-\frac{1}{144}\int_M\lvert \mathring{B}_T\rvert^4\,d\widetilde{A}-\frac{1}{24}\int_M\lvert W\rvert^2\,d\widetilde{A}+\frac{1}{12}\int_M\lvert E\rvert^2\,d\widetilde{A}\\
		&&+\int_M\lvert H_T\rvert^4\,d\widetilde{A}-2\int_M\lvert H_T\rvert^2\lambda\,d\widetilde{A}-\frac{1}{12}\int_M \left(2\lvert\mathring{B}_T\rvert^2+\Delta \lvert\mathring{B}_T\rvert^2\right)d\widetilde{A}\\
		&&+\int_M \widetilde{W}_*^2\,d\widetilde{A}-2\int_M \lambda \widetilde{W}_*\,d\widetilde{A}+2\int_M \lvert H_T\rvert^2 \widetilde{W}_*\,d\widetilde{A}\,,
	\end{eqnarray*}
	where $\chi(M)$ is the Euler characteristic, $\mathring{B}_T$ is the total trace-free second fundamental form, $W$ is the Weyl tensor of $M$, $\widetilde{W}_*$ is the average of the components of the Weyl tensor of $\widetilde{N}^m$ defined in \eqref{W*}, $\lvert E\rvert^2$ is the square of the trace-free Ricci tensor, $H_T$ is the mean curvature vector, and $\lambda$ is the scalar curvature of $M$ in $\widetilde{N}^m$. 
\end{proposition} 
\textit{Proof.} The proof of the statement follows the same steps as that of Proposition \ref{last2}. However, in this case $\widetilde{W}_*$ defined in \eqref{W*} may not be zero and, hence, we need to preserve the corresponding integral terms arising from \eqref{expandhigh}. The convergence of the first and third integrals in the last line is a consequence of the pointwise transformation of the Weyl tensor, while
$$\int_M \lambda\widetilde{W}_*\,d\widetilde{A}<\infty\,,$$
follows from the assumption $\widetilde{W}_*=\mathcal{O}(r^4)$. \hfill$\square$

\begin{proposition} Let $M$ be an asymptotically minimal (of order $2$) $4$-dimensional submanifold in $\mathbb{H}^m$. Then, the renormalized area of $M$ is given by
	\begin{eqnarray*}
		\mathcal{A}_R(M)&=&\int_M \lvert H_T\rvert^4\,dA-\int_M \lvert\overline{H}_T\rvert^4\,d\overline{A}-2\left(\int_M \lvert H_T\rvert^2 \lambda\,dA-\int_M\lvert\overline{H}_T\rvert^2\overline{\lambda}\,d\overline{A}\right)\\
		&&+\frac{1}{12}\left(\int_M \lvert E\rvert^2\,dA-\int_M \lvert\overline{E}\rvert^2\,d\overline{A}\right)-\frac{1}{12}\int_M \left(2\lvert\mathring{B}_T\rvert^2+\Delta \lvert\mathring{B}_T\rvert^2\right)dA\,,
	\end{eqnarray*}
	where $H_T$ and $\lambda$ are the mean curvature vector and the scalar curvature, $\lvert E\rvert^2$ is the square of the trace-free Ricci tensor (the Euclidean analogues are denoted with an upper bar), and $\lvert\mathring{B}_T\rvert^2$ is the square of the total trace-free second fundamental form of $M$ in $\mathbb{H}^m$.
\end{proposition}
\textit{Proof.} Apply Proposition \ref{lastnew} to the asymptotically minimal $4$-dimensional submanifold $M$ in $\mathbb{H}^m$ to get
\begin{eqnarray*}
	\mathcal{A}_R(M)&=&\frac{4}{3}\pi^2\chi(M)-\frac{1}{144}\int_M\lvert \mathring{B}_T\rvert^4\,dA-\frac{1}{24}\int_M\lvert W\rvert^2\,dA+\frac{1}{12}\int_M\lvert E\rvert^2\,dA\\
	&&+\int_M\lvert H_T\rvert^4\,dA-2\int_M\lvert H_T\rvert^2\lambda\,dA-\frac{1}{12}\int_M \left(2\lvert\mathring{B}_T\rvert^2+\Delta \lvert\mathring{B}_T\rvert^2\right)dA\,,
\end{eqnarray*}
since $\widetilde{W}_*=0$ holds when the ambient space is $\mathbb{H}^m$. Then, from \eqref{RT} and Theorem 1 of \cite{C}, we have
$$\frac{1}{144}\int_M \lvert\mathring{B}_T\rvert^4\,dA=\int_M\left(\lvert H_T\rvert^2-R_T\right)^2\,dA=\int_M \left(\lvert\overline{H}_T\rvert^2-\overline{R}_T\right)^2\,d\overline{A}\,.$$
Expanding this, applying the Chern-Gauss-Bonnet formula \eqref{CGB} to the integral of $\overline{R}_T^2=\overline{\lambda}^2$, and arguing as in Theorem \ref{t3} we conclude with the statement. \hfill$\square$

\section{Appendix: Technical Lemmas}

In this appendix, we will prove two technical lemmas employed in the proofs of Theorem \ref{t3} and Proposition \ref{last2}.

\begin{lemma}\label{Claim} Let $M$ be a hypersurface in a $5$-dimensional Poincar\'e-Einstein space $\widetilde{N}^5$ meeting the asymptotic boundary at a right angle and denote by $M_\epsilon=M\cap\{r\geq \epsilon\}$, $\epsilon>0$, the truncated hypersurface. Then, the expansion of
	$$\oint_{\partial M_\epsilon} S\,d\widetilde{s}\,,$$
in terms of $\epsilon>0$ has no constant term. (Recall that $S$ is defined in \eqref{Scurvature}.)
\end{lemma}
\textit{Proof.} Let $M_\epsilon=M\cap\{r\geq \epsilon\}$ be the truncated hypersurface of $M$, a hypersurface of the Poincar\'e-Einstein space $\widetilde{N}^5$ that meets the asymptotic boundary at a right angle.

Assume, by contradiction, that the expansion of the integral of the statement in terms of $\epsilon>0$ is
\begin{equation}\label{assumption}
	\oint_{\partial M_\epsilon} S\,d\widetilde{s}=\frac{A_3}{\epsilon^3}+\frac{A_2}{\epsilon^2}+\frac{A_1}{\epsilon}+C+\mathcal{O}(\epsilon)\,,
\end{equation}
where $A_1,A_2,A_3$ are arbitrary real numbers and $C\neq 0$ is a nonzero constant.

Consider the double of $M$ (denoted by $\widetilde{M}$) and denote by $M_\epsilon^+$ to $M_\epsilon$ and by $M_\epsilon^-$ the corresponding part on the other half of $\widetilde{M}$ such that
$$\lim_{\epsilon\to 0^+} M_\epsilon^+\cup M_\epsilon^-=\widetilde{M}\,.$$
Since $M$ meets the asymptotic boundary at a right angle, its double $\widetilde{M}$ is a sufficiently regular \emph{closed} $4$-dimensional manifold and we can employ the Chern-Gauss-Bonnet formula \eqref{CGB} to deduce that
\begin{equation}\label{ap1}
	3\int_{\widetilde{M}}\lambda^2\,d\widetilde{A}+\frac{1}{8}\int_{\widetilde{M}}\lvert W\rvert^2\,d\widetilde{A}-\frac{1}{4}\int_{\widetilde{M}}\lvert E\rvert^2\,d\widetilde{A}=4\,\pi^2\chi(\widetilde{M})\,.
\end{equation}
On the other hand, from \eqref{CGB} applied to $\widetilde{M}_\epsilon=M_\epsilon^+\cup M_\epsilon^-$ we obtain
\begin{eqnarray*}
	\frac{2A_2}{\epsilon^2}+2C+\mathcal{O}(\epsilon)&=&\oint_{\partial M_\epsilon^+}S\,d\widetilde{s}+\oint_{\partial M_\epsilon^-} S\,d\widetilde{s}=\oint_{\partial\widetilde{M}_\epsilon}S\,d\widetilde{s}\\
	&=&4\,\pi^2\chi(\widetilde{M}_\epsilon)-3\int_{\widetilde{M}_\epsilon}\lambda^2\,d\widetilde{A}-\frac{1}{8}\int_{\widetilde{M}_\epsilon}\lvert W\rvert^2\,d\widetilde{A}+\frac{1}{4}\int_{\widetilde{M}_\epsilon}\lvert E\rvert^2\,d\widetilde{A}\,.
\end{eqnarray*}
The first equality follows from \eqref{assumption} and the fact that the integrals over $\partial M_\epsilon^+$ and $\partial M_\epsilon^-$ have the same expressions after the change $\epsilon\mapsto-\epsilon$. Hence, the terms with odd powers of $\epsilon^{-1}$ cancel out, while the even ones add themselves. Finally, taking the limit as $\epsilon\to 0^+$ and comparing the finite part with that of \eqref{ap1}, we deduce that $C$ must be zero, contradicting our assumption. \hfill$\square$

\begin{lemma}\label{laplacian} Let $M$ be an asymptotically minimal (of order $2$) hypersurface in $\widetilde{N}^5$ and denote by $M_\epsilon=M\cap\{r\geq \epsilon\}$, $\epsilon>0$, the truncated hypersurface. Then, when expanding in terms of $\epsilon>0$, the constant term of
	$$\int_{M_\epsilon}2\lvert\mathring{B}\rvert^2\,d\widetilde{A}\,,$$
is the same as that of 
$$\int_{M_\epsilon}\left(2\lvert\mathring{B}\rvert^2+\Delta\lvert\mathring{B}\rvert^2\right)d\widetilde{A}\,.$$
Here, $\lvert\mathring{B}\rvert^2$ denotes the square of the trace-free second fundamental form of $M$ in $\widetilde{N}^5$. Moreover, the second expression above converges when $\epsilon\to 0^+$.
\end{lemma}
\textit{Proof.} Let $M$ be a hypersurface of the Poincar\'e-Einstein space $\widetilde{N}^5=\mathring{X}^5$ with metric $\widetilde{g}$ and consider a special defining function $r$ of the asymptotic boundary $\partial X^5$. Then, the compactified metric $\overline{g}=r^2\widetilde{g}$ of $X^5$ can be written in its normal form as
$$\overline{g}=g_r+dr^2\,,$$
where $\partial_r g_r\lvert_{r=0}=0$ (see, for instance, \cite{G,T} and references therein). Throughout this proof we will denote with upper bars the objects relative to the compactified metric $\overline{g}$. 

The proof of the lemma will be divided into two parts: first, we will show that the square of the trace-free second fundamental form $\lvert\mathring{B}\rvert^2$ of $M$ in $\widetilde{N}^5$, which is $\mathcal{O}(r^2)$ (see \eqref{equality} and \eqref{ChenB}), has no $r^3$ term when considering its expansion in terms of $r$. Then, we will employ the divergence theorem to show the claims of the statement.

Since $M$ is asymptotically minimal, $H=\mathcal{O}(r^2)$ and so, it approaches the asymptotic boundary at a right angle. Thus, near $\partial X^5$ we can see $M$ as a graph over a cylinder. Let $\mathbf{x}=(x_1,x_2,x_3)$ be local coordinates of $\partial M$ in $\partial X^5$ and consider $x_5$ to be the arc length parameter of the geodesics orthogonal to $M$ in $\partial X^5$. By locally extending these coordinates to a neighborhood of $\partial X^5$ and defining $x_4=r$, we get the local coordinates $(x_1,x_2,x_3,x_4,x_5)$ in $X^5$. Then, $M$ can be parameterized as the graph 
$$x_5=u(\mathbf{x},r)\,,$$
and the tangent space to $M$ is spanned by the vector fields
$$X_i=\partial_{x_i}+\frac{\partial u}{\partial x_i}\,\partial_{x_5}\,,$$
for every $i=1,2,3,4$.

From our choice of local coordinates, it is clear that for every $i=1,2,3,4$,
$$\bar{g}_{i5}=\bar{g}(\partial_{x_i},\partial_{x_5})=0\,,\quad\quad\quad \bar{g}_{55}=\bar{g}(\partial_{x_5},\partial_{x_5})=1\,,$$
at $\partial M$ (i.e., $r=0$ and $x_5=0$). Hence, applying this to $X_i$, we deduce that
$$u(\mathbf{x},0)=0\,,\quad\quad\quad \frac{\partial u}{\partial x_i}(\mathbf{x},0)=0\,,$$
for every $i=1,2,3,4$.

Let $\overline{\xi}$ be the unit normal to $M$ with respect to $\overline{g}$. The second fundamental form $\overline{B}$ is defined by
$$\overline{B}_{ij}=\overline{B}(X_i,X_j)=\overline{g}\left(\overline{\nabla}_{X_i}X_j,\overline{\xi}\right),$$
for every $i,j=1,2,3,4$. Here, $\overline{\nabla}$ represents the Levi-Civita connection on $X^5$. Computing the Christoffel symbols for $\overline{g}$ and using that $\partial_rg_r\lvert_{r=0}=0$, we deduce that $\overline{B}_{i4}=0$ at $r=0$ for every $i=1,2,3$. Further, differentiating $\overline{B}_{ij}$ with respect to $r$, we get $\partial_r\overline{B}_{ij}=0$ at $r=0$ for every $i,j=1,2,3$. These identities at $r=0$ follow \emph{only} from $M$ approaching the asymptotic boundary at a right angle. However, the quantities $\overline{B}_{44}=\partial_{rr}^2u$ and $\partial_r\overline{B}_{44}=\partial_{rrr}^3u$ at $r=0$ may not vanish.

We next employ the hypothesis that $H=\mathcal{O}(r^2)$ to show that $\partial_r (\lvert\overline{B}\rvert^2-4\overline{H}^2 )=0$ at $r=0$. To see this, we first consider the expansions in terms of $r$ of $\overline{H}$ and $\overline{\xi}_r=\overline{g}(\overline{\xi},\partial_r)$. The mean curvature $\overline{H}$ of $M$ with respect to the compactified metric $\overline{g}$ is defined as
\begin{equation}\label{E1}
	\overline{H}=\frac{1}{4}{\rm Trace}\, \overline{B}=\frac{1}{4}\,\overline{g}^{ij}\,\overline{B}_{ij}\,,
\end{equation}
where the Einstein summation convention is being used for the indexes $i,j=1,2,3,4$. Therefore, evaluating at $r=0$, we obtain
\begin{equation}\label{expansionH}
	\overline{H}=\frac{1}{4}\sum_{i=1}^4\overline{B}_{ii}\rvert_{r=0}+\mathcal{O}(r)\,.
\end{equation} 
Similarly, from the parameterization of $M$ as a graph given above we obtain
\begin{equation}\label{expansionXi}
	\overline{\xi}_r=\overline{g}\left(\overline{\xi},\partial_r\right)=-\overline{B}_{44}\rvert_{r=0} r+\mathcal{O}(r^2)\,.
\end{equation}
Combining \eqref{expansionH} and \eqref{expansionXi} with the analogue formula of \eqref{HRhyp-euc}, we deduce that
\begin{equation}\label{expansionHnew}
H=r\overline{H}+\overline{\xi}_r=\frac{1}{4}\left(\sum_{i=1}^3 \overline{B}_{ii}\rvert_{r=0}-3\overline{B}_{44}\rvert_{r=0}\right)r+\mathcal{O}(r^2)\,.
\end{equation}
Since $H=\mathcal{O}(r^2)$ the coefficient of $r$ must vanish. Therefore,
\begin{equation}\label{traceII}
	3\overline{B}_{44}\rvert_{r=0}=\sum_{i=1}^3\overline{B}_{ii}\rvert_{r=0}={\rm Trace}\,{\rm II}\,,
\end{equation}
where ${\rm II}$ represents the second fundamental form of $\partial M$ in $\partial X^5$ (with respect to the compactified metric $\overline{g}$). In particular, this shows that $\overline{H}=\overline{B}_{44}$ at $r=0$. By definition, $\lvert\overline{B}\rvert^2$ is given by
\begin{equation}\label{E2}
	\lvert\overline{B}\rvert^2=\overline{g}^{ij}\overline{g}^{kl}\overline{B}_{ik}\overline{B}_{jl}\,,
\end{equation}
where we are using once again the Einstein summation convention for $i,j,k,l=1,2,3,4$. Since $\partial_r\overline{g}^{ij}=0$ holds at $r=0$, it follows differentiating \eqref{E1} and \eqref{E2} and evaluating at $r=0$, that
\begin{eqnarray*}
	\partial_r\left(\lvert\overline{B}\rvert^2-4\overline{H}^2\right)\rvert_{r=0}&=&\partial_r\lvert\overline{B}\rvert^2\rvert_{r=0}-8\overline{H}\rvert_{r=0}\partial_r\overline{H}\rvert_{r=0}\\
	&=&2\sum_{i=1}^4\sum_{j=1}^4\overline{B}_{ij}\rvert_{r=0}\partial_r\overline{B}_{ij}\rvert_{r=0}-2\overline{B}_{44}\rvert_{r=0}\sum_{i=1}^4\partial_r\overline{B}_{ii}\rvert_{r=0}\\
	&=&2\overline{B}_{44}\rvert_{r=0}\partial_r\overline{B}_{44}\rvert_{r=0}-2\overline{B}_{44}\rvert_{r=0}\partial_r\overline{B}_{44}\rvert_{r=0}\\
	&=&0\,.
\end{eqnarray*}
In the second equality we have used that, as a consequence of $H=\mathcal{O}(r^2)$, \eqref{traceII} must hold and so $\overline{H}=\overline{B}_{44}$ at $r=0$. The third equality follows since $\partial_r\overline{B}_{ij}\rvert_{r=0}=0$ for every $i,j=1,2,3$, and $\overline{B}_{i4}\rvert_{r=0}=0$ for every $i=1,2,3$, as explained above.

Thus, from the standard relation between the square of the trace-free second fundamental form, and $\lvert\overline{B}\rvert^2$ and $\overline{H}^2$, we get
$$\lvert\mathring{B}\rvert^2=r^2\lvert\mathring{\overline{B}}\rvert^2=r^2\left(\lvert\overline{B}\rvert^2-4\overline{H}^2\right),$$
where the first equality follows from Theorem 1 of \cite{C} (see also \eqref{equality} and \eqref{ChenB}). Consequently, $\lvert\mathring{B}\rvert^2$ has no $r^3$ term when expanding it in terms of $r$. That is,
\begin{equation}\label{expansionB}
	\lvert\mathring{B}\rvert^2=b_2r^2+b_4r^4+\mathcal{O}(r^5)\,.
\end{equation}
Observe that from the observation of Remark \ref{final} and $\lvert B\rvert^2=\lvert\mathring{B}\rvert^2+4H^2$, the above expression \eqref{expansionB} coincides with that of (5.13) of \cite{AT}.

We are now in a position to show the statement of the lemma. The remaining part of this proof will follow the beautiful idea of Theorem 3.2 of \cite{T}. Consider the truncated hypersurfaces $M_\epsilon=M\cap \{r\geq \epsilon\}$, for $\epsilon>0$ sufficiently small. From the divergence theorem, we have
\begin{equation}\label{int1}
	\int_{M_\epsilon}\left(2\lvert\mathring{B}\rvert^2+\Delta\lvert\mathring{B}\rvert^2\right)d\widetilde{A}=2\int_{M_\epsilon}\lvert\mathring{B}\rvert^2d\widetilde{A}+\oint_{\partial M_\epsilon}\partial_\eta\lvert\mathring{B}\rvert^2d\widetilde{s}\,,
\end{equation}
where $\eta$ is the outward pointing unit conormal to $\partial M_\epsilon$ with respect to the metric $\widetilde{g}$. Hence, $\eta=-r\partial_r$. From \eqref{expansionB}, we then compute
$$\partial_\eta\lvert\mathring{B}\rvert^2=-r\partial_r\lvert\mathring{B}\rvert^2=-2b_2r^2-4b_4r^4+\mathcal{O}(r^5)\,,$$
and so
\begin{equation}\label{int2}
	\oint_{\partial M_\epsilon}\partial_\eta\lvert\mathring{B}\rvert^2d\widetilde{s}=\oint_{\partial M_\epsilon}\left(-2b_2r^{-1}-4b_4r+\mathcal{O}(r)\right)ds=-\frac{2}{\epsilon}\oint_{\partial M_\epsilon}b_2\,ds+\mathcal{O}(\epsilon)\,,
\end{equation}
since $ds=r^3d\widetilde{s}$ and $r=\epsilon$ holds along $\partial M_\epsilon$. This proves the first claim, namely, that adding the Laplacian term does not change the finite part in the expansion as $\epsilon\to 0^+$.

It now remains to show that the singular part of the integrals on the right-hand side of \eqref{int1} cancels out. Fix $r_*>0$ sufficiently small and independent of $\epsilon$. We split $M_\epsilon$ into the regions $M_\epsilon^*=M_\epsilon\cap\{r\leq r_*\}$ and $M_\epsilon\setminus M_\epsilon^*$. Then,
\begin{equation}\label{collar}
	\int_{M_\epsilon}\lvert\mathring{B}\rvert^2d\widetilde{A}=\int_{M_\epsilon^*}\lvert\mathring{B}\rvert^2d\widetilde{A}+\int_{M_\epsilon\setminus M_\epsilon^*}\lvert\mathring{B}\rvert^2d\widetilde{A}=\int_{M_\epsilon^*}\lvert\mathring{B}\rvert^2d\widetilde{A}+\widetilde{C}\,,
\end{equation}
where $\widetilde{C}$ is a constant depending only on $r_*$, and hence independent of $\epsilon$. Now, for the potentially divergent part, we apply the coarea formula to obtain
\begin{eqnarray}\label{int3}
	\int_{M_\epsilon^*}\lvert\mathring{B}\rvert^2d\widetilde{A}&=&\int_{M_\epsilon^*}\lvert\mathring{B}\rvert^2r^{-4}dA=\int_\epsilon^{r_*}\oint_{M_r}\frac{\lvert\mathring{B}\rvert^2r^{-4}}{\lvert \nabla r\rvert}dsdr=\int_\epsilon^{r_*}\oint_{M_r}\left(b_2r^{-2}+\widetilde{b}_4+\mathcal{O}(r)\right)dsdr\nonumber\\&=&\frac{1}{\epsilon}\oint_{\partial M_\epsilon} b_2\,ds+C+\mathcal{O}(\epsilon)\,,
\end{eqnarray}
where the constant $C$ depends on $r_*$, but not on $\epsilon$. Here, $\widetilde{b}_4$ depends on $b_4$ and the $r^2$ coefficient of $$\lvert\nabla r\rvert=\sqrt{1-\overline{\xi}_r^2}=1-\frac{1}{2}\overline{B}_{44}^2\rvert_{r=0}r^2+\mathcal{O}(r^3)\,.$$ 
Thus, combining \eqref{int1}, \eqref{int2}, \eqref{collar} and \eqref{int3}, we conclude with
$$\int_{M_\epsilon}\left(2\lvert\mathring{B}\rvert^2+\Delta\lvert\mathring{B}\rvert^2\right)d\widetilde{A}=2\int_{M_\epsilon^*}\lvert\mathring{B}\rvert^2d\widetilde{A}+2\widetilde{C}+\oint_{\partial M_\epsilon} \partial_\eta\lvert\mathring{B}\rvert^2d\widetilde{s}=2C+2\widetilde{C}+\mathcal{O}(\epsilon)\,,$$
which converges as $\epsilon\to 0^+$, proving the second claim. \hfill$\square$

\begin{rem}\label{final} It is worth noticing that the coefficient $b_2$ in \eqref{expansionB} is, precisely, the square of the trace-free second fundamental $\lvert\mathring{{\rm II}}\rvert^2$ of $\partial M$ in $\partial X^5$ with respect to the compactified metric $\overline{g}$. Indeed,
\begin{eqnarray*}
	b_2&=&\lvert\overline{B}\rvert^2\rvert_{r=0}-4\overline{H}^2\rvert_{r=0}=\lvert {\rm II}\rvert^2+\overline{B}_{44}^2\rvert_{r=0}-4\overline{B}_{44}^2\rvert_{r=0}=\lvert {\rm II} \rvert^2-3\overline{B}_{44}^2\rvert_{r=0}\nonumber\\
	&=&\lvert {\rm II} \rvert^2-\frac{1}{3}\left({\rm Trace} \,{\rm II}\right)^2=\lvert\mathring{{\rm II}}\rvert^2\,,
\end{eqnarray*}
where we have used that, since $H=\mathcal{O}(r^2)$ equation \eqref{traceII} holds.

In particular, we conclude that an asymptotically minimal hypersurface $M$ in $\widetilde{N}^5$ satisfies
$$\int_M\lvert\mathring{B}\rvert^2d\widetilde{A}<\infty\,,$$
if and only if $\mathring{{\rm II}}=0$ identically, which is equivalent to $\partial M$ being totally umbilical in $\partial X^5$.
\end{rem}

\section*{Acknowledgments} 

\small{The author would like to thank Aaron J. Tyrrell for fruitful discussions and for sharing ideas that have contributed to the improvement of the paper.}

\bigskip

\begin{flushleft}
	\'Alvaro P{\footnotesize \'AMPANO}\\
	Department of Mathematics and Statistics, Texas Tech University, Lubbock, TX, 79409, USA\\
	E-mail: alvaro.pampano@ttu.edu
\end{flushleft}

\end{document}